%% file: main.tex
\documentclass[3p, preprint, 11pt]{elsarticle}

\input{settings}
\begin{document}
\begin{frontmatter}
\title{
Approximate propagation of normal distributions for stochastic optimal control of nonsmooth systems%
\tnoteref{t1}
}
\tnotetext[t1]{This research was supported by DFG via project 424107692 on Robust MPC and by the EU via ELO-X 953348.}

\author[1]{Florian Messerer}
\author[1]{Katrin Baumg{\"a}rtner}
\author[1]{Armin Nurkanovi{\'c}}
\author[1,2]{Moritz Diehl}

\affiliation[1]{
   organization={Department of Microsystems Engineering (IMTEK), University of Freiburg},
   addressline={Georges-Koehler-Allee 102}, 
   city={79110 Freiburg},
   country={Germany}
   }
\affiliation[2]{
   organization={Department of Mathematics, University of Freiburg},            
   city={79104 Freiburg},
   country={Germany}
   }

\begin{abstract}
We present a method for the approximate propagation of mean and covariance of a probability distribution through ordinary differential equations (ODE) with discontinuous right-hand side.
For piecewise affine systems, a normalization of the propagated probability distribution at every time step allows us to analytically compute the expectation integrals of the mean and covariance dynamics while explicitly taking into account the discontinuity.
This leads to a natural smoothing of the discontinuity such that for relevant levels of uncertainty the resulting ODE can be integrated directly with standard schemes and it is neither necessary to prespecify the switching sequence nor to use a switch detection method.
We then show how this result can be employed in the more general case of piecewise smooth functions based on a structure preserving linearization scheme.
The resulting dynamics can be straightforwardly used within standard formulations of stochastic optimal control problems with chance constraints.
\end{abstract}

\begin{keyword}
Nonsmooth dynamics \sep stochastic optimal control \sep numerical optimal control \sep uncertain initial value



\end{keyword}

\end{frontmatter}

%%%%%%%%%%%%%%%%%%%%%%%%%%%%%%%%%%%%%%%%%%%%%%%%%%%%%%%%%%%%%%%%%%%%%%%%%%%%%%%%
\section{Introduction}
Throughout this paper, we consider initial value problems (IVP) with uncertain initial value  $x_0\in\R^n$ which are defined by an ordinary differential equation (ODE) with discontinuous right-hand side.
More specifically, we consider IVP of the form
\begin{equation} \label{eq:disc_ode}
      x(0) = x_0, \quad
      x_0 \sim \calX_0, \quad
      \dot x = f(x)
      \quad\text{with}\quad
      f(x) \defeq \left \{ \begin{array}{rl} f_1(x), & \psi(x)< 0, \\ f_2(x), & \psi(x) > 0,  \end{array} \right.
   \end{equation}
for $t\in[0, T] \eqdef \mathbb{T}$ and with the initial uncertainty described by probability distribution $\calX_0$.
The right-hand side is defined by the smooth modes $f_1$, $f_2$ and the switching function $\psi$, such that depending on the sign of $\psi(x)$ the dynamics evolve according to $f_1$ or $f_2$.
The zero-level set of $\psi$ defines the switching surface as $\mathbb{S}_\psi = \{ x\in\R^n \mid \psi(x) = 0 \}$, and for regularity we assume $\nabla\psi(x)\neq 0$ for all $x \in \mathbb{S}_\psi$.
For rigorously treating the discontinuity, i.e., the case where $\psi(x) = 0$, we refer to the notion of Filippov differential inclusions \cite{Filippov1988}.
We only consider the nondegenerate cases where the solution to \eqref{eq:disc_ode} is well defined, as will be discussed below in more detail.
The main results will be derived for the case that $f_1$, $f_2$ and $\psi$ are affine, with an extension on how these results can be employed in the nonlinear case.
Since our primary interest here is the simulation of the dynamics we consider an uncontrolled system.
However, the results straightforwardly extend to the case where $f$ additionally depends on a control input.

Nonsmooth dynamics arise in the modeling of a wide range of systems, especially in mechanics and robotics, e.g., Coulomb friction, contact models, gear boxes, but also in electrical circuits \cite{Acary2008, Brogliato2016}.
Standard theory of numerical integration of ODE is built on the assumption of Lipschitz continuity \cite{Iserles2008}. Thus, special care needs to be taken when simulating nonsmooth dynamics, for which this assumption is violated \cite{Acary2008}.
Based on a dynamic system model, optimal control provides a systematic framework for achieving a desired system behavior, i.e., optimizing trajectories based on an objective function and subject to constraints.
Its closed loop application, model predictive control (MPC), relies on the numerical solution of optimal control problems (OCP) in real time \cite{Rawlings2017}.
Often this takes the form of solving nonlinear programs (NLP) via Newton-type, i.e., derivative based, methods \cite{Nocedal2006}.
Thus, when simulating system dynamics in this context, it is relevant that the integration schemes are both efficient and provide accurate sensitivities.

Since models typically do not allow for a perfect prediction of reality, there is always some uncertainty involved.
The closely related fields of robust and stochastic optimal control try to explicitly account for this mismatch \cite{Rawlings2017, Kouvaritakis2016, Mesbah2016, Rakovic2019}.
In the robust paradigm this takes the form of set based uncertainty models, whereas the stochastic approach is concerned with probability distributions.
Numerically the resulting formulations can be very similar: 
for example, a positive definite matrix can both describe the shape of an ellipsoidal set and the covariance of a normal distribution (giving rise to ellipsoidal confidence sets), cf., e.g. \cite{Feng2020, Zanelli2021}.
Thus, while working in a stochastic framework -- the results in this paper exploit the smoothly decaying unbounded support of normal distributions -- we will still draw from results in the robust optimization literature.

In this paper we are concerned with the behavior of probability distributions under nonsmooth dynamics.
In \cite{DiMarino2016}, the authors consider distributions under the dynamics of a Moreau sweeping process. In more detail, they consider probability densities with support on a convex bounded set which moves as a function of time. This leads to nonsmooth interactions at the set boundary as the distribution is pushed along, and the authors derive results on existence and uniqueness for the resulting evolution of the distribution and provide a discretization based approximation.
In \cite{Souaiby2023} the authors consider a similar setting with an additional drift term.
Based on a Lipschitz approximation of the nonsmooth dynamics, they describe the evolution of the distribution via the Liouville equation.
Leveraging an additional finite order moment approximation allows them to compute the distribution over time.
In \cite{Kirches2006}, the authors propose the linearization based propagation of an uncertainty set for an ODE with discontinuous right-hand side.
Based on detection of the switch they account for the change in integrator sensitivity resulting from the discontinuity via the so called jump matrix, cf. \cite{Bock1987, Stewart2010}. This allows them to solve an OCP with nonsmooth dynamics under parametric uncertainty.
For an overview of jump matrices in this context, see also \cite{Kong2023}.

\subsection{Contribution and outline}
We present a method for the approximate propagation of mean and covariance of a probability distribution through an ODE in the form of \eqref{eq:disc_ode}. 
The method is based on
(a) linearizing the right-hand side function at the current mean in terms of its components $f_1$, $f_2$, $\psi$, such that a piecewise affine function is obtained which preserves the discontinuous structure of $f$,
(b) approximating the current probability distribution by a normal distribution which for piecewise affine $f$ allows us to analytically compute the current change of mean and covariance. However, since this neglects the change of the higher-order moments, this does not lead to an exact propagation.
Finally, we demonstrate how the derived dynamics can be used in a stochastic optimal control problem formulation with chance constraints.

We start by discussing the relevant background on ODE with discontinuous right-hand side in Section~\ref{sec:ode_discont} and on IVP with uncertain initial value in Section~\ref{sec:ivp_unc}.
In Section~\ref{sec:switch_const_1D} we provide a detailed discussion of the simplified setting of a scalar piecewise constant system to further the intuitive understanding.
This is followed by the derivation of the main result for piecewise affine systems in Section~\ref{sec:switch_aff_nD}, the extension to the piecewise smooth case in Section~\ref{sec:pw_smooth}, and its application to stochastic OCP in Section~\ref{sec:stochOCP}, with a concluding Section~\ref{sec:conclusions}.

\subsection{Notation}
For a multivariate function $f\colon\R^n \to \R^m$, $x\mapsto f(x)$, the gradient is defined as the transpose of the Jacobian, $\nabla f(x) = \dpartial{f(x)}{x}^\top$, such that $\nabla f(x) \in \R^{n\times m}$.
For two vectors $x\in\R^n$, $y\in \R^m$ their vertical concatenation is denoted by $(x,y) \defeq [x^\top, y^\top]^\top$.
The convex hull of two vectors $x, y \in \R^{n}$ is $\conv(x, y) \defeq \{ (1-\theta) x + \theta y \mid \theta \in [0,1]  \}.$
If a symmetric matrix $S = A + A^\top$ is a sum of a matrix $A\in\R^{n\times n}$ and its transpose, we abbreviate this as $S = A + (\star)$.

\section{Ordinary differential equations with discontinuous right-hand side}
\label{sec:ode_discont}

In this section we briefly summarize the relevant background on ODE with discontinuous right-hand side of the form \eqref{eq:disc_ode}, to the extent relevant in the context of numerical optimal control.
For a more in-depth discussion we refer especially to \cite{Stewart2010}. 

\subsection{Nondegenerate switching cases}
We distinguish between two nondegenerate switching cases \cite{Stewart2010}, in which the solution trajectory is well defined:
\begin{enumerate}
   \item Crossing the discontinuity: $\psi(x) = 0$ and both $\nabla \psi(x)^\top f_1(x) > 0$ and $\nabla \psi(x)^\top f_2(x) > 0$.
   In consequence $\dtotal{\psi(x(t))}{t} > 0$ immediately before and after the switching time such that the state immediately leaves the switching surface after reaching it, and crosses from $\psi(x) < 0$ to $\psi(x) > 0$.
   The case where the surface is crossed in the opposite direction is analogous.
   \item Sliding mode (or trapped case):  $\psi(x) = 0$ and $\nabla \psi(x)^\top f_1(x) > 0$ but $\nabla \psi(x)^\top f_2(x) < 0$. 
   In this case, the solution will remain on the surface, i.e., $\dtotal{\psi(x(t))}{t} = 0$ after the switching time.
\end{enumerate}

\subsection{Numerical integration} \label{subsec:nonsmooth_numerical_int}
Because $f$ is discontinuous, the standard theory of ODE and their numerical integration does not hold since it is built on the assumption of (Lipschitz)-continuity of $f$ \cite{Iserles2008}.
Thus, when relying on standard integration schemes and their theory, the switch needs to be considered explicitly.
Otherwise the standard results on the order of integration accuracy will not hold: In general, for a step size $h$ the error will be $\bigO(h)$ irrespective of the order of the scheme \cite{Acary2008}. Furthermore, the error in the integration map sensitivities will even be independent of the step size \cite{Stewart2010, Nurkanovic2020}.
In the context of optimal control this requires either a predefined switching sequence or a switch detecting integration scheme which supplies correct sensitivities \cite{Acary2008, Stewart1990a,Nurkanovic2022}.
An intuitive and common workaround is to smoothen the right-hand side \eqref{eq:disc_ode}.
This results in the smooth approximate dynamics
\begin{equation} \label{eq:f_smoothed}
   f_\sigma(x) = (1 - \alpha_\sigma(\psi(x))) f_1(x) + \alpha_\sigma(\psi(x)) f_2(x)),
\end{equation}
where $\alpha_\sigma\colon\R\to\R$ is a smooth approximation of the Heaviside step function, parametrized by $\sigma > 0$ and with increasing accuracy as $\sigma \to 0$.
One choice is $\alpha_\sigma(\xi) = (1+\tanh(\xi / \sigma)) / 2$.
However, the resulting ODE will be increasingly nonlinear and stiff for decreasing $\sigma$ and require decreasingly smaller step sizes for sufficiently precise integration.
Furthermore, it can be shown that the integrator has to have a step size of $h=\smallO(\sigma)$ in order for its sensitivities to be adequate \cite{Stewart2010, Nurkanovic2020}, which is an important requirement if it is to be used as a component in the formulation of a nonlinear program (NLP).
Thus, one is in general well advised to carefully choose an appropriate integration scheme when handling an ODE with discontinuous right-hand side.

\subsection{Solution sensitivities}
While the solution maps of IVP with discontinuous right-hand side are continuous, their sensitivities may have jumps when encountering a switch.
Denote by $x(t; x_0)$, $t\in\mathbb{T}$, the solution of IVP \eqref{eq:disc_ode} for a given initial value $x_0$.
Assume the IVP is not initialized at a switch but that at some time $ 0 < t_\mathrm{s} < T$ with $x_\mathrm{s}\defeq x(t_\mathrm{s}; x_0)$ the solution reaches the switching surface, $\psi(x_\mathrm{s}) = 0$ and $\psi(x(t, x_0)) \neq 0$ for all $t\in[0,t_\mathrm{s})$.
Then the sensitivity of the final state with respect to the initial value is given by
\begin{equation}
   \dpartial{x(T; x_0)}{x_0} = \dpartial{x(T - t_\mathrm{s}; x_\mathrm{s})}{x_\mathrm{s}} J(x_\mathrm{s}) \dpartial{x(t_\mathrm{s}; x_0)}{x_0},
\end{equation}
where the jump matrix $J(x_\mathrm{s})$ accounts for the jump of sensitivity and is given by
\begin{equation} \label{eq:saltation_mat}
   J(x_\mathrm{s}) = \eye +  \frac{(f_2(x_\mathrm{s}) - f_1(x_\mathrm{s}))\nabla\psi(x_\mathrm{s})^\top}{ \nabla\psi(x_\mathrm{s})^\top f_1(x_\mathrm{s})   }
   \quad\mathrm{resp.}\quad
   J(x_\mathrm{s}) = \eye +  \frac{(f_2(x_\mathrm{s}) - f_1(x_\mathrm{s}))\nabla\psi(x_\mathrm{s})^\top}{ \nabla\psi(x_\mathrm{s})^\top( (f_1(x_\mathrm{s}) - f_2(x_\mathrm{s}) )   }
\end{equation}
for the crossing resp. sliding mode \cite{Stewart2010}.
Without loss of generality, the jump matrix for the crossing case is given under the assumption that $\psi(x_0) < 0$, i.e., the state crosses from mode $f_1$ into mode $f_2$.

\begin{example}[Crossing the discontinuity] \label{ex:crossing1D_nom}
   Consider a scalar system with state $x\in\R$ and $ \dot x = 3$ for $x < 0$ and $\dot x = 1$ for $x > 0$.
   We simulate the trajectory $x(t; x_0)$ for $t\in[0,T]$, and for three different values of the initial state $x_0$, given by the set $\{ \bar\mu_1, \bar\mu_1+3\sigma_1, \bar\mu_1-3\sigma_1\}$, with $\bar\mu_1 = -3$, $\sigma_1=0.3$, and $T=2$.
   The resulting trajectories as well as the corresponding integrator map $x(x_0, T)$ are shown in Fig.~\ref{fig:crossing1D_nominal}.
   We observe that the discontinuity in $\dot x$ at $x=0$ leads to a kink in the trajectories. 
   Due to this kink, the distance between the trajectories narrows.
   Whereas initially the distance between the two outer points is $6\sigma_1$, after each point has crossed the switch this distance has narrowed to $ \tfrac{\bar f_2}{\bar f_1}6\sigma_1 =  \tfrac{1}{3} 6\sigma_1$. The scaling factor of $\tfrac{1}{3}$ corresponds to the slope of the integrator map in the respective region and is given by the jump matrix \eqref{eq:saltation_mat}.
\end{example}

   \begin{figure}
      \centering
      \includegraphics[width=.49\textwidth]{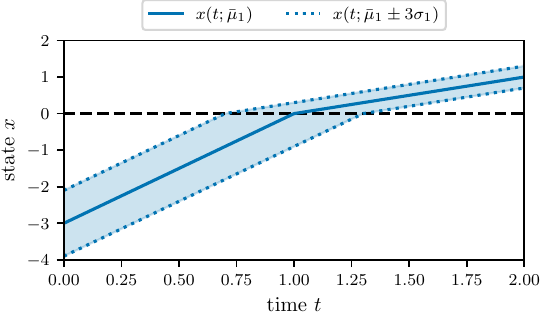}
      \includegraphics[width=.49\textwidth]{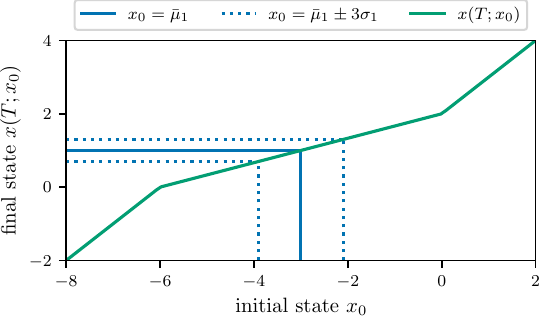}
      \caption{
         Crossing the discontinuity.
      Left: The state trajectories from Example~\ref{ex:crossing1D_nom}.
      The switch at $x=0$ leads to a kink in each trajectory and to a scaling of their distances with respect to each other.
      Right: The corresponding integrator map $x(T; x_0)$ as a function of $x_0$, which is piecewise affine.
      For $-6\leq x_0 \leq 0$, the discontinuity is crossed within the integration interval.
      The blue lines visualize the corresponding mapping of the initial states from the left-hand side plot. The scaling of the distances is a consequence of the slope of the integrator map in the corresponding region.
      }
      \label{fig:crossing1D_nominal}
   \end{figure}

\begin{example}[Sliding mode]\label{ex:sliding1D_nom}
   Now consider a system with state $x\in\R$ and $ \dot x = 3$ for $x < 0$ and $\dot x = -1$ for $x > 0$.
   This results in a system with sliding mode, i.e., once a trajectory reaches $x=0$, it stays there.
   Again, we simulate the trajectory $x(t; x_0)$ for $t\in[0,T]$, and for three different values of the initial state $x_0$, given by $\{ \bar\mu_1, \bar\mu_1+3\sigma_1, \bar\mu_1-3\sigma_1\}$, this time with $\bar\mu_1 = -1$, $\sigma_1=0.6$, and $T=1.2$.
   The resulting trajectories as well as the corresponding integrator map $x(x_0, T)$ are shown in Fig.~\ref{fig:sliding1D_nominal}.
   The sliding mode leads to a narrowing of the distance between the trajectories over time, until all of them have reached $x=0$.
   This results in the zero slope of the integrator map $x(T; x_0)$ in the region of initial states $x_0$ such that $x(T; x_0) = 0$.
   Note that this time the zero sensitivity is not a consequence of the jump matrix \eqref{eq:saltation_mat}, but of the resulting sliding mode dynamics $\dot x = 0$ if $x=0$.
   This is due to the scalar state space.
   In a higher dimensional state space, the sensitivity would not necessarily be zero, since the state could still evolve on the switching surface with nontrivial dynamics.
\end{example}
\begin{figure}
   \centering
   \includegraphics[width=.49\textwidth]{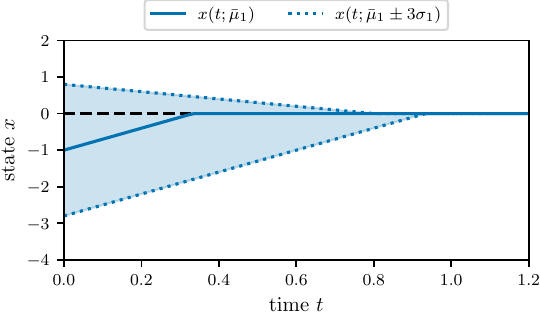}
   \includegraphics[width=.49\textwidth]{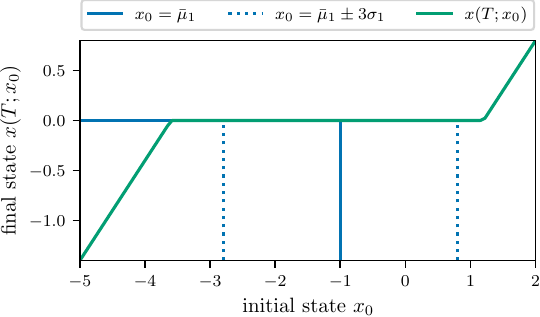}
   \caption{
   Sliding mode.   
   Left: The state trajectories from Example~\ref{ex:sliding1D_nom}.
   Once a trajectory reaches $x=0$ it stays there, which leads to narrowing of the distances over time.
   Right: The corresponding integrator map $x(T; x_0)$ as a function of $x_0$, which is piecewise affine.
   The flat region in the center corresponds to the values of $x_0$ such that $x(T, x_0)$ reaches the switching surface.
   The blue lines visualize the corresponding mapping of the initial states from the left-hand side plot.
   }
   \label{fig:sliding1D_nominal}
\end{figure}

\section{Initial value problems with uncertain initial value}
\label{sec:ivp_unc}

We now consider again a general IVP with $x(0)=x_0$, $\dot x = f(x)$, $t\in \mathbb{T}$, and assume its solution is well defined for every $x_0\in\R^{n}$. However, the initial state is not exactly known.
Instead it follows a probability distribution, $x_0 \sim \calX_0$.
If we denote by $x(t; x_0)$ the solution to the IVP after the time interval $[0, t]$ given $x(0)=x_0$, this induces a distribution $\calX(t)$ over the solution trajectory such that $x(t; x_0)\sim\calX(t)$.
Alternatively we can view this as an IVP in distribution space, $\calX(0)=\calX_0$, $\dot \calX = \calF(\calX)$, with appropriately defined $\calF$.

In principle, we can describe this evolution in terms of the moments of $\calX$, assuming that all moments are finite and uniquely determine $\calX$.
This holds if $\calX$ has bounded support or if its tails decay sufficiently fast, which includes normal distributions \cite{Billingsley1995}.
In particular, consider the first and second-order moments, mean and covariance, defined as
\begin{subequations}
   \begin{align}
      \mean_\calX &\defeq \EXV_{x\sim \calX}\{ x \}  \label{eq:def_mean}
      \\ 
      \cov_\calX &\defeq \EXV_{x\sim \calX}\{ (x-\mean_\calX)(x-\mean_\calX)^\top \}
      \label{eq:def_cov}
   \end{align}
\end{subequations}
Noting that $\EXV_{x\sim \calX(t)}\{ x \} =  \EXV_{x_0\sim \calX_0}\{ x(t; x_0) \}$, we get the time derivative of the mean as
\begin{subequations} \label{eq:moment_dot}
\begin{equation} 
   \dtotaltfrac{}{t}\mean_{\calX(t)} = \dtotaltfrac{}{t}  \EXV_{x_0\sim \calX_0} \{ x(t; x_0) \}
   = \EXV_{x_0\sim \calX_0}\{ \dtotaltfrac{}{t} x(t; x_0) \}
   = \EXV_{ x \sim \calX(t)}\{ f(x) \}. \label{eq:mu_dot}
\end{equation}
   Similarly, we get the time derivative of the covariance as
   \begin{align}
   \dtotaltfrac{}{t} \cov_{\calX} &= \dtotaltfrac{}{t}  \EXV_{x\sim \calX}\{ (x-\mean_\calX)(x-\mean_\calX)^\top \} \\
  &=
   \EXV_{x\sim \calX}\{ (f(x)-\dtotaltfrac{}{t} \mean_\calX)(x-\mean_\calX)^\top + (\star)  \} \\
  &=  \EXV_{x\sim \calX}\{ f(x)(x-\mean_\calX)^\top  \}+ (\star). \label{eq:Sigdot_2}
\end{align}
\end{subequations}
However, to exactly describe the evolution of $\calX$, we would need to consider the change of all moments up to infinite order, which is in general intractable.
Further, we cannot treat the expectation of a nonlinear transformation of a random variable in this general setting.

\subsection{Linearization-based uncertainty propagation for smooth dynamics}
For IVP with smooth dynamics (or in general for smooth nonlinear transformations of random variables) a standard approach to approximate the uncertainty propagation is based on linearization.
Examples include widely used methods such as the Extended Kalman Filter \cite{Stengel1986}.
There are two major variants: (i) linearization of the right-hand side of the ODE, and (ii) linearization of the integration map.

In the first variant, after substituting $f$ by its first-order Taylor approximation at the mean, the expectations in \eqref{eq:moment_dot} can be analytically computed as
\begin{subequations} \label{eq:moment_dot_taylor}
   \begin{align} 
      \dtotal{}{t}\mean_\calX &\approx \EXV_{ x \sim \calX}\{ f(\mean_\calX) + \dpartial{f(\mean_\calX)}{\mean_\calX}(x-\mean_\calX) \} = f(\mean_\calX),\\
      \dtotal{}{t} \cov_\calX &\approx \EXV_{x\sim \calX}\{ \big(f(\mean_\calX) + \dpartial{f(\mean_\calX)}{\mean_\calX}(x-\mean_\calX)\big)(x-\mean_\calX)^\top   \} + (\star) = \dpartial{f(\mean_\calX)}{\mean_\calX} \cov_\calX + (\star). 
      \label{eq:moment_dot_taylor_cov}
   \end{align}
\end{subequations}
When computing the expectation in \eqref{eq:moment_dot_taylor_cov}, the first term vanishes due to $\EXV_{x\sim\calX} \{ x - \mean_\calX \} = 0$, cf. \eqref{eq:def_mean}.
The second term yields the covariance \eqref{eq:def_cov} premultiplied by a constant (with respect to the expectation operator).
The approximation in \eqref{eq:moment_dot_taylor} is exact for the case that $f$ is an affine function.
The resulting approximate propagation of the first two moments is defined by the IVP
\begin{subequations} \label{eq:musigdot_standard}
   \begin{align}
      \mu(0) &= \mean_{\calX_0}, & \dot \mu &= f(\mu), \\
      \Sigma(0) &= \cov_{\calX_0}, & \dot\Sigma &= \dpartial{f(\mu)}{\mu} \Sigma + \Sigma \dpartial{f(\mu)}{\mu}^\top.
   \end{align}
\end{subequations}
Unsurprisingly, the covariance dynamics are in the form of the continuous time differential Lyapunov equation \cite{Soederstroem2002}, i.e., the covariance dynamics of a linear system.

While in the preceding approach the uncertainty propagation is given in continuous time, in direct optimal control the dynamics are usually considered only on a discrete time grid.
Thus an alternative approach, corresponding to the second variant, is to discretize the dynamics first and only then consider uncertainty.
Let $f_h(x_k) \defeq x(h; x_k)$ denote the integration of the dynamics over the discretization time step $h$ given the initial value $x_k$, such that $x_{k+1} = f_h(x_k)$.
Note that in practice, this integration map is typically approximated by standard numerical integration schemes, cf., e.g., \cite{Rawlings2017}.
Then, given $x_k \sim \calX_k$, we can approximate the expectations as
\begin{subequations} 
   \begin{align} 
      \mean_{\calX_{k+1}} & = \EXV_{ x \sim \calX_k}\{ f_h(x) \} 
      \\&\approx \EXV_{ x \sim \calX_k}\{ f_h(\mean_{\calX_k}) + \dpartial{f_h(\mean_{\calX_k})}{\mean_{\calX_k}}(x-\mean_{\calX_k}) \} = f_h(\mean_{\calX_k}),\\
       \cov_{\calX_{k+1}} &=
       \EXV_{x\sim \calX_k}\{ (f_h(x) - \mean_{\calX_{k+1}})(f_h(x) - \mean_{\calX_{k+1}})^\top  \}\\
       &\approx \EXV_{x\sim \calX_k}\{ \dpartial{f_h(\mean_{\calX_k})}{\mean_{\calX_k}}(x-\mean_{\calX_k})(x-\mean_{\calX_k})^\top\dpartial{f_h(\mean_{\calX_k})}{\mean_{\calX_k}}^\top \} = \dpartial{f_h(\mean_{\calX_k})}{\mean_{\calX_k}} \cov_{\calX_k} \dpartial{f_h(\mean_{\calX_k})}{\mean_{\calX_k}}^\top.
   \end{align}
\end{subequations}
This results in the approximate propagation 
\begin{subequations} \label{eq:musigplus_standard}
   \begin{align}
      \mu_0 &= \mean_{\calX_0}, & \mu_{k+1}&= f_h(\mu_k), \\
      \Sigma_0 &= \cov_{\calX_0}, & \Sigma_{k+1} &= \dpartial{f_h(\mu_k)}{\mu_k} \Sigma_k \dpartial{f_h(\mu_k)}{\mu_k}^\top,
   \end{align}
\end{subequations}
with the covariance dynamic corresponding to the discrete time  Lyapunov difference equation.
Thus, it is sufficient to obtain a (numerically approximated) map for the deterministic IVP, from which the covariance dynamics directly follow.
This is in contrast to \eqref{eq:musigdot_standard}, where both mean and covariance need to be numerically integrated.
One important practical consideration is that during the numerical integration of \eqref{eq:musigdot_standard} it can happen that $\Sigma$ takes indefinite values, while the propagation in \eqref{eq:musigplus_standard} guarantees that $\Sigma_k \succeq 0$ if $\Sigma_0 \succeq 0$, cf. \cite{Gillis2013}.

\subsection{Linearization-based uncertainty propagation for nonsmooth dynamics}
Both of the approaches discussed in the previous subsection rely on linearization of a smooth function. Thus, they are not directly applicable to ODE with a discontinuous right-hand side.
Again, an intuitive approach is to smoothen the dynamics as in \eqref{eq:f_smoothed} and then apply a linearization based scheme.
This does not come without issues: all the problems regarding stiffness and accuracy of derivatives from the nominal case will transfer such that it is challenging to use within an optimization problem, cf. Section~\ref{subsec:nonsmooth_numerical_int}.
Thus, as in the nominal case, an alternative is to treat the switch explicitly and propagate the covariance based on \eqref{eq:musigplus_standard}.
Using the jump matrix \eqref{eq:saltation_mat} the sensitivity jump can be handled explicitly.
However, in the context of optimal control this requires either a method for switch detection \cite{Kirches2006} or a predefined switching sequence.

Furthermore, the linearization based propagation of mean and covariance leans on the assumption that the nonlinear dynamics are sufficiently well approximated by a linearization at the mean within the region of uncertainty. 
For a switched system this is clearly not the case if the mean is on one side of the switch but a nonnegligible amount of probability mass on the other.
While this works for some situations, in others it can cause a complete failure of the uncertainty propagation as we will see in a later example.

\subsection{Normalization-based uncertainty propagation}

In the remainder of this paper we will derive an alternative method for the approximate propagation of mean and covariance.
Instead of linearization, the method is based on ``normalization'' of the probability distribution, i.e., at each point in time we approximate the true distribution by a normal distribution that is defined by our current value for mean and variance.
For switched affine systems this will allow us to compute the expectations in 
\eqref{eq:moment_dot} analytically, resulting in a natural smoothing of the discontinuity.
This yields an easy to implement ODE for mean and variance that can be treated by standard integrators and be straightforwardly used within tractable stochastic OCP formulations.

A similar idea of renormalization after every time step (although in a discrete time setting) is used in the method of moment matching for recursive time series prediction based on Gaussian process models: for these, the mean and covariance of the output can be exactly computed given a normal distribution in the input variable \cite{Quinonero-Candela2003}.
Further, the unscented Kalman filter \cite{Julier2004} discretizes the current normal distribution into systematically chosen samples which are then propagated through the nonlinear function. Based on mean and variance of the propagated samples, a new normal distribution is obtained.

While we will explain our suggested approach in detail in the following three sections, we already derive some results that will be useful later on.

Consider a variable $z\in\R^n$, two distributions $\calZ_1, \calZ_2$, and a function $g:\R^n \to \R^m$.
We would like to approximate $\EXV_{z\sim\calZ_1}\{g(z)\}$ by computing instead
the expectation with respect to $\calZ_2$. 
We can write
\begin{subequations} \label{eq:error_def}
\begin{equation} \label{eq:error_def_mean}
   \EXV_{z\sim\calZ_1} \{ g(z) \} =
   \EXV_{z\sim\calZ_2} \{ g(z) \} + 
   \underbrace{\EXV_{z\sim\calZ_1} \{ g(z) \} - \EXV_{z\sim\calZ_2} \{ g(z) \}}_{\displaystyle \eqdef \delta_g(\calZ_1, \calZ_2)},
\end{equation}
where $\delta_g(\calZ_1, \calZ_2)$ defines the resulting error.
Clearly, if the distributions are identical, this error is zero.
Intuitively, the error becomes larger as the two distribution become more dissimilar.
In the remainder of this paper we will use this definition to describe the error resulting from our approximation, but  provide no rigorous analysis.
However, we point out that the definition of the error is closely related to integral probability metrics \cite{Mueller1997}.
For example, in case that $g$ is Lipschitz continuous,  $\delta_g(\calZ_1, \calZ_2)$ can be bounded via the dual representation of the 1-Wasserstein (Kantorovich-Rubinstein) distance of $\calZ_1$ and $\calZ_2$, cf. \cite[Remark 6.5]{Villani2009}.
Similar bounds may be derived for the case that $g$ is piecewise Lipschitz, based on conditional expectations.
Additionally, we define
\begin{align}  \label{eq:error_def_cov_half}
   \Delta^\prime_g(\calZ_1, \calZ_2) &\defeq  \EXV_{z\sim\calZ_1} \{ g(z)(z-\mean_{\calZ_1})^\top \} - \EXV_{z\sim\calZ_2} \{ g(z)(z-\mean_{\calZ_2})^\top\}
   \\ \label{eq:error_def_cov}
   \Delta_g(\calZ_1, \calZ_2) &\defeq \Delta^\prime_g(\calZ_1, \calZ_2) + \Delta^\prime_g(\calZ_1, \calZ_2)^\top
\end{align}
\end{subequations}
with a similar motivation.

Now consider again the state distribution $\calX(t)$ as defined from the IVP with uncertain initial value  $x(0)=x_0$, $x_0\sim\calX_0$, $\dot x = f(x)$, $t\in \mathbb{T}$.
Both computing and representing $\calX(t)$ is in general intractable, and we would like to approximate it by a normal distribution $\calX(t) \approx \calN(\mu(t), \Sigma(t))$, parameterized by $\mu$ and $\Sigma$, and with its probability density function (PDF) given by
\begin{equation}
  \tilde\phi(x; \mu, \Sigma) = \frac{1}{\sqrt{(2\pi)^n\det\Sigma}} \expe^{-\frac{1}{2} (x-\mu)^\top \Sigma^{-1} (x-\mu)}.
\end{equation}
Our aim is to derive a tractable IVP for the parameters $\mu$ and $\Sigma$,
\begin{subequations}
   \begin{alignat}{6}
      \mu(0) &= \bar\mu_0,&\quad
      \dot \mu &= \hat f_\mu(\mu, \Sigma),
      \\
      \Sigma(0) &= \bar\Sigma_0,&\quad
      \dot \Sigma &= \hat f_\Sigma(\mu, \Sigma),
   \end{alignat}
\end{subequations}
such that $\calX(t) \approx \calN(\mu(t), \Sigma(t))$.

As a first step, we consider what happens when we replace the expectation with respect to $\calX$ by an expectation with respect to $\calN(\mu,\Sigma)$ in the moment dynamics \eqref{eq:moment_dot}.
\begin{lemma} \label{lem:error_from_wrong_expectation}
   Consider a distribution in state space, $x\sim\calX$, that evolves according to $\dot x = f(x)$.
   Consider also a normal distribution on the same space, $\calN(\mu,\Sigma)$.
   Then, the time derivative of mean and covariance can be written as
   \begin{subequations} \label{eq:mom_dot_2} 
   \begin{alignat}{6} 
      \dtotal{}{t}\mean_{\calX} &=  \EXV_{ x \sim \calN(\mu, \Sigma)}\{f(x) \} && + \delta_f(\calX, \calN(\mu,\Sigma)) , \label{eq:mu_dot_2}
      \\ \label{eq:sig_dot_2}
      \dtotal{}{t}\cov_{\calX} &= \EXV_{x\sim\calN(\mu,\Sigma)}\{ f(x)(x-\mu)^\top + (x-\mu)f(x)^\top \} && + \Delta_f(\calX, \calN(\mu,\Sigma)),
   \end{alignat}
   \end{subequations}
   with the error terms  $\delta_f$, $\Delta_f$ defined in \eqref{eq:error_def}.
\end{lemma}
\begin{proof}
   The mean dynamics $\dtotal{}{t} \mean_{\calX}$ we have from \eqref{eq:mu_dot}. Exchanging the expectation over $\calX$ by the expectation over $\calN(\mu,\Sigma)$ as in \eqref{eq:error_def_mean} results in \eqref{eq:mu_dot_2}.

   From \eqref{eq:Sigdot_2} we have the covariance dynamics as
   \begin{equation} \label{eq:covdot_proof_useful_lemma}
         \dtotal{}{t} \cov_{\calX} =  
         \EXV_{x\sim \calX}\{ f(x)(x-\mean_{\calX})^\top\} + (\star).
   \end{equation}
   We rearrange the first term of the above right-hand side as
   \begin{subequations}
      \begin{align}
         &\EXV_{x\sim\calX}\{ f(x)(x-\mean_\calX)^\top \}
         \\&\quad
         =\phantom{-}
         \EXV_{x\sim\calN(\mu,\Sigma)}\{ f(x)(x-\mean_\calX +\mu-\mu)^\top \} + \EXV_{x\sim\calX}\{ f(x)(x-\mean_\calX)^\top \}
         \\&\quad\phantom{=}\;-
         \EXV_{x\sim\calN(\mu,\Sigma)}\{ f(x)(x-\mean_\calX)^\top \}
         \\&\quad=
         \EXV_{x\sim\calN(\mu,\Sigma)}\{ f(x)(x-\mu)^\top \} + \EXV_{x\sim\calX}\{ f(x)(x-\mean_\calX)^\top \} -  \EXV_{x\sim\calN(\mu,\Sigma)}\{ f(x)(x-\mu)^\top \}
         \\&\quad=
         \EXV_{x\sim\calN(\mu,\Sigma)}\{ f(x)(x-\mu)^\top \} + \Delta^\prime_f(\calX, \calN(\mu,\Sigma)),
         \end{align}
   \end{subequations}
   where in the first step we add and subtract the expectation over $\calN(\mu, \Sigma)$ as in \eqref{eq:error_def_mean} and also add and subtract  $\mu^\top$.
   In the second step we cancel  $\EXV_{x\sim\calN(\mu,\Sigma)}\{ f(x)\mean_\calX^\top \} - \EXV_{x\sim\calN(\mu,\Sigma)}\{ f(x)\mean_\calX^\top \}$,
   and in the third step we use the definition of $\Delta^\prime_f(\calX, \calN(\mu,\Sigma))$ from \eqref{eq:error_def_cov_half}.
   Repeating the reformulation for the second term in \eqref{eq:covdot_proof_useful_lemma}, which is the transpose of the first, and substituting the definition of $\Delta_f$ from \eqref{eq:error_def_cov} yields \eqref{eq:sig_dot_2}.
\end{proof}

We obtain a differential equation for $\mu$ and $\Sigma$ by disregarding the error terms in \eqref{eq:mom_dot_2}.
This results in
\begin{subequations} \label{eq:ode_mu_sig}
   \begin{alignat}{6}
      \dot \mu &= \hat f_\mu(\mu, \Sigma) &
      \quad\text{with}\quad
      \hat f_\mu(\mu, \Sigma) &\defeq \EXV_{ x \sim \calN(\mu, \Sigma)}\{f(x) \},
      \\ \label{eq:ode_mu_sig_sig}
      \dot \Sigma &= \hat f_\Sigma(\mu, \Sigma) & 
      \quad\text{with}\quad
      \hat f_\Sigma(\mu, \Sigma) &\defeq  \EXV_{x\sim  \calN(\mu, \Sigma)}\{ f(x) (x-\mu)^\top\} + (\star).
   \end{alignat}
\end{subequations}

In general, \eqref{eq:ode_mu_sig} is still intractable due to the expectation over a nonlinear function.
However, for the piecewise constant and piecewise affine systems treated in this paper, we can compute this expectation analytically, as will be derived in the following two sections.
For this purpose, the following lemma will be useful. 
It shows that we get a closed-form expression of the variance dynamics $\dot \Sigma$ for free if we can find a closed-form expression of the mean dynamics  $\dot \mu$.

\begin{lemma} \label{lem:Sig_dot_lazy}
   Consider $\hat f_\mu(\mu, \Sigma)$ and $\hat f_\Sigma(\mu, \Sigma)$ as defined in \eqref{eq:ode_mu_sig}.
   The following holds:
   \begin{equation} \label{eq:Sig_dot_lazy}
      \hat f_\Sigma(\mu, \Sigma)  = \dpartial{\hat f_\mu(\mu, \Sigma)}{\mu} \Sigma + \Sigma\dpartial{\hat f_\mu(\mu, \Sigma)}{\mu}^\top.
   \end{equation}
\end{lemma}
\begin{proof}
   From \eqref{eq:ode_mu_sig_sig} we have  $\hat f_\Sigma(\mu, \Sigma) = \EXV_{x\sim  \calN(\mu, \Sigma)}\{ f(x) (x-\mu)^\top\} + (\star)$.
   We rearrange the first term as
\begin{subequations} 
   \begin{align} 
      \EXV_{x\sim  \calN(\mu, \Sigma)}\{ f(x) (x-\mu)^\top\}  
      &= \int_{\R^n} f(x) (x-\mu)^\top\tilde\phi(x; \mu,\Sigma) \Sigma^{-1} \Sigma \mathrm{d}x\\
      &= \left( \int_{\R^n} f(x)\tilde\phi(x; \mu,\Sigma)(x-\mu)^\top \Sigma^{-1} \mathrm{d}x \right) \Sigma  \label{eq:before_partial_deriv}\\
      &= \left(\int_{\R^n} f(x)\left(\dpartial{}{\mu}\tilde\phi(x; \mu,\Sigma)\right)  \mathrm{d}x \right) \Sigma  \label{eq:after_partial_deriv}\\
      &= \left(\dpartial{}{\mu} \int_{\R^n} f(x)\tilde\phi(x; \mu,\Sigma) \mathrm{d}x \, \right) \Sigma \\
      &= \left(\dpartial{}{\mu}  \EXV_{x\sim  \calN(\mu, \Sigma)} \{f(x) \}\right) \Sigma \\
      &=  \dpartial{\hat f_\mu(\mu, \Sigma)}{\mu} \Sigma,
   \end{align}
   \end{subequations}
where from \eqref{eq:before_partial_deriv} to  \eqref{eq:after_partial_deriv} we used
$
   \dpartial{}{\mu}\tilde\phi(x; \mu,\Sigma) =\tilde\phi(x; \mu,\Sigma) (x-\mu)^\top \Sigma^{-1}
$.
Noting that the second term in \eqref{eq:ode_mu_sig_sig}, as indicated by the $(\star)$, is simply the transpose of the first, results in \eqref{eq:Sig_dot_lazy}.
\end{proof}

We conclude this section with a small lemma which we will use several times throughout the paper. Before stating the lemma, we define the probability density function of the univariate standard normal distribution as
   \begin{equation} \label{eq:standard_normal_pdf_cdf}
      \phi(\nu) \defeq \tilde\phi(\nu; 0, 1) = \tfrac{1}{\sqrt{2\pi}}\expe^{-\frac{1}{2}\nu^2}, \quad
      \Phi(\nu) = \int_{-\infty}^\nu \phi(\nu') \mathrm{d}\nu',
   \end{equation}
   with associated cumulative distribution function (CDF) $\Phi$.
While there exists no closed-form expression for $\Phi$, we can in practice treat it as such since numerical implementations are readily available via the error function $\erf(\cdot)$.

\begin{lemma} \label{lem:int_aff_norm}
   Let $\alpha, \beta, \xi, \bar\xi, \mu, \sigma \in\R$, $\sigma > 0$.
   Then
\begin{subequations}
\begin{align}
          \label{eq:int_aff_norm_1}
            \int_{-\infty}^{\bar \xi} (\alpha \xi + \beta ) \tilde\phi(\xi; \mu, \sigma^2) \mathrm{d}\xi
            &=
            -\alpha \sigma^2 \tfrac{1}{\sigma} \phi(\tfrac{\bar \xi - \mu}{\sigma}) + (\alpha \mu + \beta)\Phi(\tfrac{\bar \xi - \mu}{\sigma})  \\
      \label{eq:int_aff_norm_2}
         \int_{\bar \xi}^\infty (\alpha \xi + \beta ) \tilde\phi(\xi; \mu, \sigma^2) \mathrm{d}\xi
            &= \phantom{-}
            \alpha  \sigma^2 \tfrac{1}{\sigma}  \phi(\tfrac{\bar \xi - \mu}{\sigma}) + (\alpha \mu + \beta)(1 - \Phi(\tfrac{\bar \xi - \mu}{\sigma}))
\end{align}
\end{subequations}
\end{lemma}
\begin{proof}
   For \eqref{eq:int_aff_norm_1} we have
   \begin{subequations}
    \begin{align}
      \int_{-\infty}^{\bar \xi} (\alpha \xi + \beta ) \tilde\phi(\xi; \mu, \sigma^2) \mathrm{d}\xi 
      &=  \int_{-\infty}^{\frac{\bar \xi - \mu}{\sigma}} (\alpha (\sigma\nu + \mu) + \beta ) \phi(\nu) \mathrm{d}\nu
      \\
      &=  \alpha \sigma\int_{-\infty}^{\frac{\bar \xi - \mu}{\sigma}} \nu  \phi(\nu) \mathrm{d}\nu
      +( \alpha \mu + \beta)\int_{-\infty}^{\frac{\bar \xi - \mu}{\sigma}}  \phi(\nu) \mathrm{d}\nu
      \\
      &=
      -\alpha \sigma^2 \tfrac{1}{\sigma} \phi(\tfrac{\bar \xi - \mu}{\sigma}) + (\alpha \mu + \beta)\Phi(\tfrac{\bar \xi - \mu}{\sigma}),
   \end{align}
\end{subequations}
where in the first line we substituted $\xi = \sigma \nu + \mu$.
For \eqref{eq:int_aff_norm_2} the derivation is similar.
Note that we do not simplify $\sigma^2 \tfrac{1}{\sigma}$ here to emphasize the appearance of $\tfrac{1}{\sigma} \phi(\tfrac{\bar \xi - \mu}{\sigma}) = \tilde\phi(\bar\xi; \mu, \sigma^2)$.
\end{proof}

\section{Uncertainty propagation for scalar piecewise constant systems}
\label{sec:switch_const_1D}

In order to get an intuitive understanding of how a normal distribution behaves when encountering a discontinuity, we will first limit our analysis to a scalar state space $x\in\R$ with piecewise constant dynamics $\dot x =f(x)$ of the form
\begin{equation} \label{eq:pw_const_1D}
      f(x) = \left \{ \begin{array}{rl} \bar f_1, & x < 0, \\\bar f_2, & x > 0,  \end{array} \right.
\end{equation}
with $\bar f_1, \bar f_2 \in \R\setminus\{0\}$.
If $\bar f_1$ and $\bar f_2$ have the same sign, this leads to a crossing of the discontinuity, whereas $\bar f_1 > 0$ and $\bar f_2 < 0 $ leads to the sliding mode.
In the following we will take a detailed look at how normal distributions behave when propagated through such a system: first for an example with a crossing of the discontinuity, then for an example with a sliding mode.
This is then followed by an approximate approach for propagating its first two moments.

\subsection{Case study: the switched normal distribution }
We start by revisiting Example~\ref{ex:crossing1D_nom}.
   \begin{figure}
      \centering
      \includegraphics[width=\columnwidth]{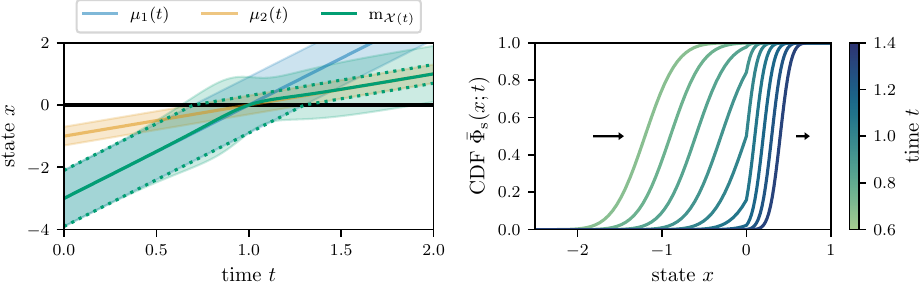}
      \caption{
         Left: The means of the two imagined normal distributions $\calN(\mu_i(t), \sigma_i^2)$, $i=1,2$, compared to the mean of the exactly propagated distribution. The shaded regions indicate $3\sigma$ on each side of the mean. The dotted lines indicate the 99.7\% probability mass corresponding to the original $\pm3\sigma$ region.
         Right: The cumulative density function  $\bar\Phi_\mathrm{s}(x; \mu_1(t), \mu_2(t), \sigma_1, \sigma_2)$ of the switched normal distribution for various time points as the distribution crosses the switch at $x=0$.
         The arrows indicate the state dynamics.
         }
      \label{fig:crossing1D}
   \end{figure}
\begin{example}[Crossing the discontinuity (cont.)] \label{ex:crossing}
   Consider again the system from Example~\ref{ex:crossing1D_nom}, i.e., dynamics of the form \eqref{eq:pw_const_1D} with $\bar f_1 = 3$ and $\bar f_2 = 1$.
   This time, assume that the initial value $x_0$ follows a normal distribution, $x_0\sim \calN(\bar\mu_1, \sigma_1^2)$, with $\bar\mu_1 = -3$, $\sigma_1 = 0.3$.
   For ease of presentation, we first limit our attention to the interval $x_0 \in [\bar\mu_1 - 3\sigma_1, \bar\mu_1 + 3\sigma_1]\eqdef \bar X_0$, noting that it contains around 99.7\% of the probability mass. 
   Consistent with the behavior we saw in Example~\ref{ex:crossing1D_nom}, after each point of the interval has passed through the discontinuity, e.g., at time $t=T$, the distribution has been scaled by factor  $\tfrac{\bar f_2}{\bar f_1} = \tfrac{1}{3}$, cf. Fig~\ref{fig:crossing1D_nominal}.
   Correspondingly, for $x_0\in\bar X_0$, the evolved state $x(T; x_0)$ is distributed like $\calN( x(T; \bar\mu_1), \sigma_2)$ with
   $\sigma_2 = \tfrac{\bar f_2}{\bar f_1}\sigma_1 = \tfrac{1}{3} \sigma_1$.
   This holds, because $x(T; x_0)$ is an affine map with respect to $\bar X_0$, cf. Fig~\ref{fig:crossing1D_nominal} (right).
   If the interval has not yet fully crossed the discontinuity, only the points that did are scaled correspondingly.

   Since $\bar\mu_1$ is the median of the initial distribution, with $50\%$ of probability mass on each side of $\bar\mu_1$, it follows that, as the system evolves, $x(t; \bar\mu_1)$ will always be the median of the evolved distribution.
   However, since the points above $x(t; \bar\mu_1)$ are squeezed together earlier than those below, the mean will be below the median while the distribution is crossing the discontinuity, cf.  Fig.~\ref{fig:crossing1D} (left).
   
   We observe that that there are two virtual normal distributions involved, visualized in Fig.~\ref{fig:crossing1D} (left).
   The first one is associated with points below the switch and given by $\calN(\mu_1(t), \sigma_1^2)$, with the evolution of its mean defined by $\mu_1(0) = \bar\mu_1$ and $\dot \mu_1 = \bar f_1$, such that $\mu_1(t) = \bar \mu_1 + \bar f_1 t$.
   The second distribution is $\calN(\mu_2(t), \sigma_2^2)$ with $\dot \mu_2 = \bar f_2$.
   Further, there is a time point $t_\mathrm{s}$ such that $\mu_1(t_\mathrm{s}) = \mu_2(t_\mathrm{s}) = 0$.
   In the considered example, this is $t_\mathrm{s} = 1$.
   From these conditions we can compute $\bar \mu_2 = \tfrac{\bar f_2}{\bar f_1} \bar \mu_1$ such that $\mu_2(t) = \bar \mu_2 + \bar f_2 t$.

   We now revisit the assumption that $x_0 \sim \calN(\bar\mu_1, \sigma_1^2)$.
   While this is certainly a valid assumption, it is slightly inconsistent:
   It makes a statement about all points, including those above the switching surface.
   But we have seen that as soon as an interval of points passes the discontinuity, their distribution is scaled such that it will be closely related to $\calN(\mu_2(t), \sigma_2^2)$.
   With this in mind, it seems more natural to assume that the distribution for all points above the switch has already been transformed.
   In consequence, we can describe the distribution at time $t$ as
   \begin{equation} \label{eq:switched_normal_dist}
      x(t) \sim \calN_\mathrm{s}(\mu_1(t), \mu_2(t), \sigma_1^2, \sigma_2^2),
   \end{equation}
   where $\mu_1(t)$, $\mu_2(t)$, $\sigma_1$, $\sigma_2$ are as defined above, and which is defined via its PDF
   \begin{equation} \label{eq:switched_normal_distribution}
      \bar\phi_{\mathrm{s}}(x; \mu_1, \mu_2, \sigma_1, \sigma_2) \defeq
      \left\{\begin{array}{rl}
         \tfrac{1}{\sigma_1}\phi( \tfrac{x-\mu_1}{\sigma_1} ), & x < 0, \\ 
         \tfrac{1}{\sigma_2}\phi( \tfrac{x-\mu_2}{\sigma_2} ), & x > 0, \\ 
      \end{array} 
      \right.
   \end{equation}
   with associated CDF $\bar\Phi_\mathrm{s}$.
   We refer to this distribution as a switched normal distribution, since depending on the sign of $x$, this distribution switches between $\calN(\mu_1(t), \sigma_1^2)$ and $\calN(\mu_2(t), \sigma_2^2)$.
   We note that the above definition of $\bar\phi_{\mathrm{s}}$ does not result in a probability distribution for arbitrary parameter values, since its integral over $\R$ is not necessarily given by 1.
   A visualization of this distribution can be found in Fig.~\ref{fig:crossing1D} (right).
\end{example}

We now consider how a normal distribution behaves for the sliding mode, revisiting Example~\ref{ex:sliding1D_nom}.
\begin{example}[Sliding mode (cont.)] \label{ex:sliding}
   Consider the system from Example~\ref{ex:sliding1D_nom}, i.e., dynamics of the form \eqref{eq:pw_const_1D} with $\bar f_1 = 3$ and $\bar f_2 = -1$.
   We want to describe the evolution of $x(t; \bar x_0)\sim\calX(t)$ in the case that $x_0$ follows a normal distribution, $x_0\sim \calN(\bar\mu_0, \sigma^2)$, with $\bar\mu_0 = -1$, $\sigma = 0.6$.
   On both sides of the origin, all points are transported towards the discontinuity, at speeds $\lvert \bar f_1\rvert$ resp. $\lvert \bar f_2\rvert$.
   Points that have reached the origin stay there.

   Again we imagine two virtual normal distributions, $\calN(\mu_1(t), \sigma^2)$ and $\calN(\mu_2(t), \sigma^2)$, with $\mu_i(t) = \bar \mu_0 + \bar f_i t$ for $i=1,2$, visualized in Fig.~\ref{fig:sliding1D} (left).
   Initially, they are identical, $\mu_1(0) = \mu_2(0) = \bar\mu_0$. As time passes, they are transported according to either $\dot x = f_1$ or $\dot x = f_2$.
   Comparing this to the evolution of $\calX(t)$, we see that probability mass below the discontinuity behaves like $\calN(\mu_1(t), \sigma^2)$ whereas above the discontinuity it behaves like $\calN(\mu_1(t), \sigma^2)$.
   The remaining probability mass accumulates at the origin.
   In other words, for $x < 0$ the PDF of $\calX(t)$ is given by the PDF of $\calN(\mu_1(t), \sigma^2)$, whereas for $x > 0$ it corresponds to $\calN(\mu_2(t), \sigma^2)$.
   At $x=0$, a Dirac delta $\delta(x)$ accounts for the remaining probability mass.
   Based on the switched normal distribution defined in \eqref{eq:switched_normal_distribution}
   we can describe this as
   \begin{equation} \label{eq:switched_normal_dist_var}
      x(t) \sim \calN_\mathrm{s}^\prime(\mu_1(t), \mu_2(t), \sigma^2, \sigma^2),
   \end{equation}
   with PDF given by
   \begin{equation} \label{eq:switched_normal_distribution_var}
      \bar\phi_{\mathrm{s}}^\prime(x; \mu_1, \mu_2, \sigma_1, \sigma_2) \defeq \bar\phi_{\mathrm{s}}(x; \mu_1, \mu_2, \sigma_1, \sigma_2)
      + P_0(\mu_1, \mu_2, \sigma_1, \sigma_2)  \delta(x),
   \end{equation}
   where
   \begin{equation}
      P_0(\mu_1, \mu_2, \sigma_1, \sigma_2) \defeq \Phi(-\tfrac{\mu_1}{\sigma_1}) + 1 - \Phi(-\tfrac{\mu_2}{\sigma_2}) 
   \end{equation}
   is the probability mass accumulated at the origin.
   The corresponding CDF is visualized in Fig.~\ref{fig:sliding1D} (right).
   We point out that $\bar\phi_{\mathrm{s}}$ from \eqref{eq:switched_normal_distribution} is the special case of $ \bar\phi_{\mathrm{s}}^\prime$ where  $P_0 = 0$.
   Further, \eqref{eq:switched_normal_distribution_var} only defines a probability distribution for $P_0 \geq 0$.
\end{example}

\begin{figure}
   \centering
   \includegraphics[width=\columnwidth]{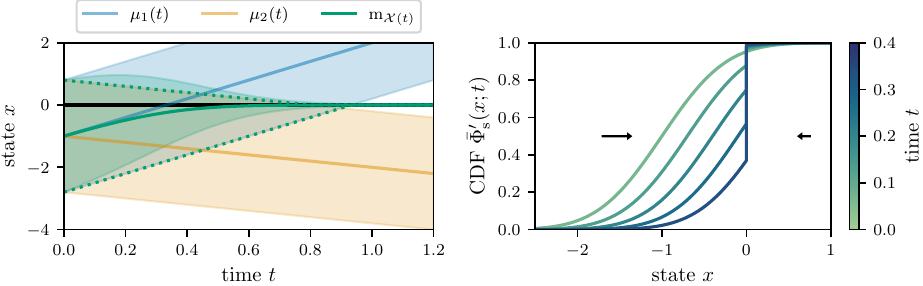}
   \caption{
      Left: The means of the two imagined normal distributions $\calN(\mu_i(t), \sigma_i^2)$, $i=1,2$, compared to the mean of the exactly propagated distribution. The shaded regions indicate $3\sigma$ on each side of the mean. The dotted lines indicate the 99.7\% probability mass corresponding to the original $\pm3\sigma$ region.
      Right: The cumulative density function  $\bar\Phi^\prime_\mathrm{s}(x; \mu_1(t), \mu_2(t), \sigma_1, \sigma_2)$ of the modified switched normal distribution for various time points as the distribution crosses the switch at $x=0$.
      On both sides of $x=0$, the probability mass is transported towards the origin, where it accumulates.
      The arrows indicate the state dynamics.
      }
   \label{fig:sliding1D}
\end{figure}

\subsection{Approximate uncertainty dynamics via normalization}

Even though we can derive analytical results for the specific case considered in the previous subsection, we can see that already in such a simple scenario this becomes rather involved. 
If we want to consider more general situations, the above results are of limited use.
A large part of the complexity came from the fact that the state was not distributed normally.
In the following we will see what happens if at each time we approximate the true distribution by a normal distribution.
This will allow us to derive an explicit ODE for mean and variance.

\begin{proposition}
   Consider the IVP with state $x\in\R$, $x(0) = x_0$, $\dot x = f(x)$, $t\in\mathbb{T}$, with piecewise constant $f$ of the form \eqref{eq:pw_const_1D}, and with $x_0 \sim\calN(\mu, v)$ where $v=\sigma^2$.
   Then for the corresponding IVP in $\mu$ and $v$ as defined by \eqref{eq:ode_mu_sig}, such that $\calN(\mu(t), v(t))$ is an approximation of the exact distribution at time $t$, we can state the right-hand side explicitly as
\begin{subequations} \label{eq:dyn_approx_1Dconst}
      \begin{align} 
         \dot \mu &=  \bar f_1 \Phi(\tfrac{-\mu}{\sqrt{v}}) + \bar f_2 (1 - \Phi(\tfrac{-\mu}{\sqrt{v}})), \label{eq:dyn_approx_1Dconst_mean}
         \\
         \dot v &= 2(\bar f_2 - \bar f_1) \tfrac{1}{\sqrt{v}}\phi(-\tfrac{\mu}{\sqrt{v}}) v. \label{eq:dyn_approx_1Dconst_var}
      \end{align}
\end{subequations}
\end{proposition}
\begin{proof}
   For the mean dynamics we have from \eqref{eq:mu_dot} that
\begin{subequations}
      \begin{align}
      \dot \mu = \EXV_{x\in\calN(\mu, v)} \{ f(x) \} 
      &= \int_{-\infty}^{0} \bar f_1 \tfrac{1}{\sqrt{v}}\phi(\tfrac{x -\mu}{\sqrt{v}}) \mathrm{d}x + \int_{0}^{\infty} \bar f_2 \tfrac{1}{\sqrt{v}}\phi(\tfrac{x -\mu}{\sqrt{v}}) \mathrm{d}x
      \\
      &= \bar f_1 \Phi(\tfrac{-\mu}{\sqrt{v}}) + \bar f_2 (1 - \Phi(\tfrac{-\mu}{\sqrt{v}})),
      \end{align}
\end{subequations}
   where in the last step we used Lemma~\ref{lem:int_aff_norm}. The variance dynamics follow from Lemma~\ref{lem:Sig_dot_lazy}.
   Adapting the multivariate notation of Lemma~\ref{lem:Sig_dot_lazy} to the currently considered univariate case, \eqref{eq:Sig_dot_lazy} reads as $\dot v = 2\dpartial{\hat f_\mu(\mu, v)}{\mu} v$, with $\hat f_\mu(\mu, v)$ given by the right-hand side of \eqref{eq:dyn_approx_1Dconst_mean}.
   Additionally noting that $\dpartial{}{\mu} \Phi(\tfrac{-\mu}{\sqrt{v}}) = - \frac{1}{\sqrt{v}} \phi(\tfrac{-\mu}{\sqrt{v}})$, cf. \eqref{eq:standard_normal_pdf_cdf}, yields \eqref{eq:dyn_approx_1Dconst_var}.
\end{proof}

As it turns out, the mean dynamics \eqref{eq:dyn_approx_1Dconst_mean} are exactly a form of smoothing of the form \eqref{eq:f_smoothed}, with smoothing function $\alpha_\sigma(\xi) = \Phi(\tfrac{\xi}{\sigma})$. This is visualized in Fig.~\ref{fig:crossing1D_approx}.
As long as uncertainty is sufficiently large, $\sigma\gg 0$, the discontinuity is strongly smoothed, such that most standard methods of numerical integration are well suited.
However, if $\sigma$ decreases, the approximate dynamics become increasingly stiff and nonlinear, such that more care needs to be taken.
As we have seen, this is especially relevant in the context of sliding modes, since they inherently lead to a decay of uncertainty.
A practical remedy in this context could be the consideration of process noise, since this would in effect provide a lower bound on the state covariance.
This is however beyond the scope of this paper.

Having pointed out these limitations, we will now try to get an understanding of the error resulting from the proposed approximation. 
For this purpose, we perform the following numerical experiments.

\begin{figure}
   \centering
   \includegraphics[width=\columnwidth]{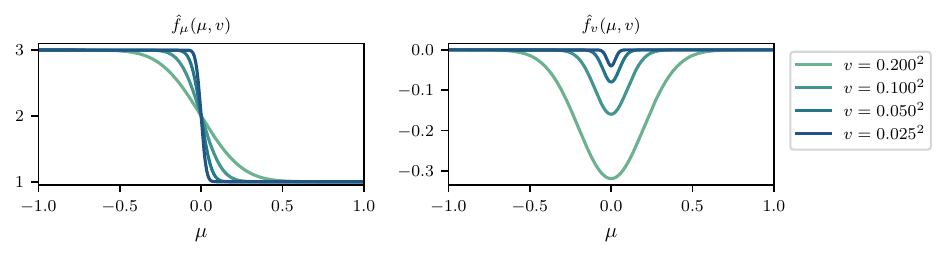}
   \caption{The approximated mean and variance dynamics \eqref{eq:dyn_approx_1Dconst} for the system from Example~\ref{ex:crossing}.}
   \label{fig:crossing1D_approx}
\end{figure}

\begin{example}
We apply dynamics \eqref{eq:dyn_approx_1Dconst} to the system from Example~\ref{ex:crossing} with initial distribution $x_0\sim\calN(\bar \mu_1, \sigma_1)$ and compare the resulting evolution of $\mu(t)$ and $\sigma^2(t)$ to mean and variance of the exactly propagated switched normal distribution \eqref{eq:switched_normal_dist}.
While the initial distribution is not exactly identical to a switched normal distribution, the mean is at a distance of more than 6 standard deviations from the switching surface, such that the initial probability mass on the other side of the switching surface is less than $10^{-9}$, and difference between the two distributions is negligible.
The resulting errors over time are shown in Fig.~\ref{fig:crossing1D_error}.
Some error is accumulated while the distribution is crossing the switch, but the error is constant before and after.

We now perform two experiments on the influence of parameters on the total error that is accumulated during the crossing of the switch.
Strictly speaking the switching process is never completed nor does it have a well defined starting point due to the unbounded support of the normal distributions. 
We use the error at time $T$ as a proxy for the total error while ensuring that both initial and final mean have a distance of at least six standard deviations from the switching surface in their respective direction.
Thus, the probability mass which has not yet switched at the end can be neglected.
In the first experiment, we vary the initial standard deviation $\sigma_0$, with fixed $\bar f_1 =  3$, $\bar f_2 = 1$, $T=2$.
The results are shown in Fig.~\ref{fig:crossing1D_error_exp} (left).
In the second experiment we vary the pre-switching dynamics $\bar f_1$ while keeping the post-switching dynamics fixed at $\bar f_2 = 1$, with initial standard deviation $\sigma_0 = 0.3$ and $T=4$.
For $\bar f_1 = \bar f_2$ we would expect no error since in this case there is no discontinuity, and an error that continuously rises as $\lvert \bar f_1 - \bar f_2 \rvert $  moves away  from zero.
Thus, we vary $f_1$ both in the interval $[\bar f_2 +10^{-12}, \bar f_2 + 10^{-1}]$ and $[\bar f_2 - 10^{-12}, \bar f_2 - 10^{-1}]$.
The results are shown in Fig.~\ref{fig:crossing1D_error_exp} (right).
In combination these experiments suggest that the final integration error is of order $\bigO(\sigma_0 \lvert  \bar f_1 - \bar f_2\rvert)$.
While we do not investigate this in detail here, we point out that this is consistent with the error resulting from smoothing the discontinuity in the nominal case, cf.~\cite{Stewart2010}.
\end{example}

\begin{figure}
   \centering
   \includegraphics[width=\columnwidth]{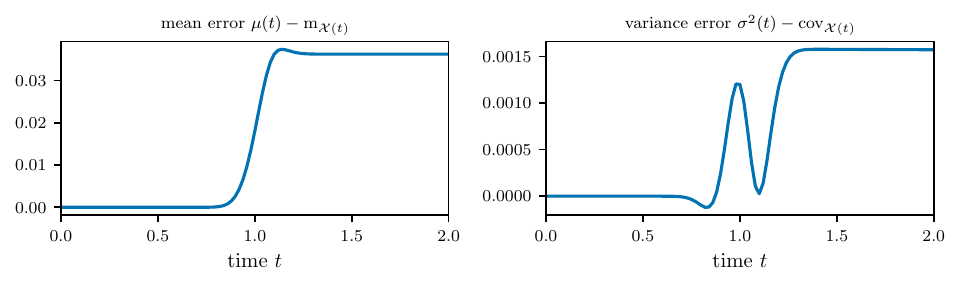}
   \caption{Left: error between exact and approximated mean over time.
   Right: error between exact and approximated variance over time.}
   \label{fig:crossing1D_error}
\end{figure}
\begin{figure}
   \centering
   \includegraphics[width=\columnwidth]{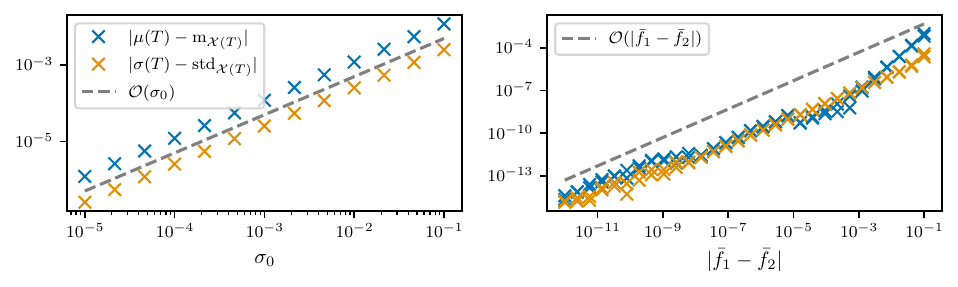}
   \caption{Left: Final integration error as a function of initial variance.
   Right: Final integration error as a function of the jump in the dynamics.
   The plot includes both the case where $\bar f_1 - \bar f_2$ is positive and where it is negative.
   }
   \label{fig:crossing1D_error_exp}
\end{figure}

\section{Uncertainty propagation for piecewise affine systems}
\label{sec:switch_aff_nD}

After having developed an understanding of how a normal distribution behaves when passing through the discontinuity of an ODE with piecewise constant right-hand side, we now consider the more general case of a piecewise affine system with affine switching function,
\begin{equation} \label{eq:f_pw_aff}
   f(x) \defeq \left \{ \begin{array}{rl} A_1 x + \bar f_1, & g^\top(x-\bar x)< 0, \\A_2 x + \bar f_2, & g^\top(x-\bar x) > 0.\end{array} \right.
\end{equation}
This corresponds to the specific of \eqref{eq:disc_ode} where $f_1(x) = A_1 x +\bar f_1$, $f_2(x) = A_2 x +\bar f_2$ and $\psi(x) = g^\top(x - \bar x)$, with $A_1, A_2 \in \R^{n\times n}$, $f_1, f_2, g, \bar x \in \R^n$.
In this case the switching surface is a hyper plane $\{ x\in\R^n \mid g^\top(x-\bar x) = 0 \} $ with normal vector $g$ and $\bar x$ an arbitrary point on the surface.

While in the piecewise constant case a normal distribution is asymptotically recovered after crossing a discontinuity, this is not generally true in the piecewise affine case. 
However, similarly to the switched constant case, during the switching process the distribution will be strongly different from a normal distribution, but after crossing the resulting distribution will resemble a slightly perturbed normal distribution.

For this case we will not try to find an exact parameterization of the distribution as it encounters the switching surface.
Instead, we will directly derive approximate dynamics for mean and variance and validate them by sampling.
Before stating these dynamics we define by $\tilde\phi_{g}(\xi; \mu, \Sigma)$ the PDF of the univariate normal distribution obtained by projecting the $n$-variate distribution $\calN(\mu, \Sigma)$ onto the direction $g\in\R^n$,
\begin{equation}
   \tilde\phi_{g}(\xi; \mu, \Sigma) \defeq \frac{1}{\sqrt{g^\top \Sigma g}} \phi\left(\frac{\xi - g^\top \mu}{ \sqrt{g^\top \Sigma g}}\right),
\end{equation}
where $\xi \in\R$ is the remaining degree of freedom after projection, and
which has mean $\mu_g \defeq g^\top \mu$, standard deviation $\sigma_g \defeq \sqrt{g^\top \Sigma g}$ and associated CDF $\tilde\Phi_g(\xi; \mu, \Sigma)$.
Similarly, the projection of $\bar x$ is denoted by $\bar x_g \defeq g^\top \bar x$

\begin{proposition} \label{prop:musigdot_swi_aff_approx}
   Consider an IVP with state $x\in\R^n$, $x(0) = x_0$, $\dot x = f(x)$, $t\in\mathbb{T}$, $x_0 \sim\calN(\mu, \Sigma)$, and with $f(x)$ piecewise affine with affine switching function as in \eqref{eq:f_pw_aff}.
   Then for the corresponding IVP in $\mu$ and $\Sigma$ as defined by \eqref{eq:ode_mu_sig}, such that $\calN(\mu(t), \Sigma(t))$ is an approximation of the exact distribution at time $t$, we can state the right-hand side explicitly as
   \begin{equation} \label{eq:dyn_swi_aff_approx}
      \dot \mu =  \hat f_\mu(\mu, \Sigma),
      \quad
      \dot \Sigma =\underbrace{\dpartial{ \hat f_\mu(\mu, \Sigma)}{\mu} \Sigma + \Sigma\dpartial{ \hat f_\mu(\mu, \Sigma)}{\mu} ^\top}_{\displaystyle = \hat f_\Sigma(\mu, \Sigma)}
   \end{equation}
   with 
   \begin{subequations} \label{eq:musig_dot_pw_aff}
   \begin{align}
      \label{eq:f_mu}
       \hat f_\mu(\mu, \Sigma) &= (A_2 - A_1)\Sigma g  \tilde\phi_g(\bar x_g; \mu, \Sigma )
      + f_1(\mu)  \tilde\Phi_g(\bar x_g; \mu, \Sigma )
      + f_2(\mu) (1 - \tilde\Phi_g(\bar x_g; \mu, \Sigma )),\\
      \label{eq:Sig_mu}
      \hat f_\Sigma(\mu,\Sigma) &=
      (A_2 - A_1) \tfrac{ \Sigma g g^\top \Sigma}{g^\top\Sigma g} (\bar x_g - g^\top \mu)  \tilde\phi_g(\bar x_g; \mu, \Sigma ) )+ (f_2(\mu) - f_1(\mu)) g^\top \Sigma \tilde\phi_g(\bar x_g; \mu, \Sigma )\\
      &\qquad+ A_1\Sigma  \tilde\Phi_g(\bar x_g; \mu, \Sigma ) + A_2\Sigma (1- \tilde\Phi_g(\bar x_g; \mu, \Sigma )) + (\star).  \nonumber
   \end{align}
   \end{subequations}
\end{proposition}
\begin{corollary}
   Consider the same situation as in Prop. \ref{prop:musigdot_swi_aff_approx} and additionally assume $A_1=A_2\eqdef A$.
   Then
   \begin{subequations}
      \begin{align}
      \label{eq:f_mu_cor}
       \hat f_\mu(\mu, \Sigma) &= A \mu + \bar f_1 \tilde\Phi_g(\bar x_g; \mu, \Sigma )
      + \bar f_2 (1 - \tilde\Phi_g(\bar x_g; \mu, \Sigma )),\\
      \label{eq:Sig_mu_cor}
      \hat f_\Sigma(\mu,\Sigma) &=
      (\bar f_2 - \bar f_1) g^\top \Sigma \tilde\phi_g(\bar x_g; \mu, \Sigma )
     + A\Sigma + (\star). 
   \end{align}
   \end{subequations}
\end{corollary}
\begin{proof}
   Without loss of generality we state the proof under the assumption that the switching function gradient is given by $g=(1,0,\dots,0)$, such that the first dimension of the state space is orthogonal to the switching surface.
   If this assumption does not hold one may substitute $y = Rx$, with invertible $R\in\R^{n \times n}$, such that in the transformed state space the assumption holds.
   We use Lemma~\ref{lem:error_from_wrong_expectation} and \ref{lem:Sig_dot_lazy} to derive \eqref{eq:musig_dot_pw_aff}.
   The basic structure of the proof is to separate the expectation in \eqref{eq:mu_dot_2} into two steps:
   the expectation over the first dimension, in which the switch occurs, and the conditional expectation over the remaining dimensions, with respect to which the dynamics are linear.
   We partition the state space as
   \begin{equation}
      x\sim\calN(\mu, \Sigma), \;
      x = \begin{bmatrix} x_1 \\ x_2 \end{bmatrix}, \; \mu = \begin{bmatrix} \mu_1 \\ \mu_2  \end{bmatrix}, \;
      \Sigma = \begin{bmatrix} \sigma_g^2 & \Sigma_{12} \\ \Sigma_{21} & \Sigma_{22} \end{bmatrix},
   \end{equation}
   where $x_1, \mu_1 \in \R$, $x_2, \mu_2 \in \R^{n-1}$, $\sigma_g^2 \in \R$, $\Sigma_{22} \in \SSS^{n-1}$ and $\Sigma_{21}, \Sigma_{12}$ correspondingly. 
   We use $\sigma_g^2 = g^\top \Sigma g $ instead of $\Sigma_{11}$ to emphasize that in $x_1$ we have a univariate (normal) distribution resulting from projecting $\calN(\mu, \Sigma)$ onto the direction $g$.

From \eqref{eq:mu_dot_2} we obtain the mean dynamics as
\begin{subequations} \label{eq:f_mu_proof_part_1}
   \begin{align} 
      \hat f_\mu(\mu, \Sigma) 
      &= \EXV_{x\sim\calN(\mu, \Sigma)} \{ f(x) \} 
      = \EXV_{x_1}\{  \EXV_{x_2 \mid x_1} \{f( x)\}  \} \\&
      = \EXV_{x_1}\{ f\left( \EXV_{x_2 \mid x_1} \{x\}\right)  \} 
      = \EXV_{x_1\sim\calN( \mu_1,  \sigma_g^2 )}\{ f(\breve \mu(x_1) \}
      \end{align}
\end{subequations}
where we first use conditional expectation and then move the expectation over $x_2$ into $f(\cdot)$ which is allowed since the switch depends only on $x_1$ such that $f(\cdot)$ is affine in $x_2$.
For the inner expectation we use $\EXV_{x_2 \mid x_1} \{x_1\} = x_1$ and $\EXV_{x_2 \mid x_1} \{x_2\} = \mu_2 + \Sigma_{21}\sigma_g^{-2}(x_1 -  \mu_1)$ due to $x\sim\calN(\mu, \Sigma)$ such that over all
   \begin{equation}
      \breve \mu(x_1) \defeq
      \EXV_{x_2 \mid x_1} \{x\} =
      \begin{bmatrix} 1 \\ \Sigma_{21}\sigma_g^{-2} \end{bmatrix} x_1 + \begin{bmatrix} 0 \\  \mu_2 -  \Sigma_{21} \sigma_g^{-2} \mu_1\end{bmatrix}.
   \end{equation}
Defining
\begin{equation}
      a_i \defeq A_i  \begin{bmatrix} 1 \\ \Sigma_{21}\sigma_g^{-2}\end{bmatrix},
      \quad 
      b_i \defeq A_i  \begin{bmatrix} 0 \\  \mu_2 -  \Sigma_{21} \sigma_g^{-2} \mu_1\end{bmatrix} + \bar f_1, \quad i=1,2.
\end{equation}
   allows us to write 
\begin{equation}
   f(\breve\mu(x_1)) = 
      \left \{ \begin{array}{rl} a_1 x_1 + b_1, & x_1 - \bar x_1 < 0, \\ a_2 x_1 + b_2, &x_1 - \bar x_1  > 0,  \end{array} \right.
\end{equation}
where $\bar x_1 \in \R$ results from the partitioning of $\bar x = (\bar x_1, \bar x_2)$.
Continuing \eqref{eq:f_mu_proof_part_1}, we get
\begin{subequations}
   \begin{align}
      \hat f_\mu(\mu, \Sigma)
      &=  \EXV_{x_1\sim\calN( \mu_1,  \sigma_g^2 )}\{  f(\breve\mu(x_1))\}\\
      &=
      \int_{-\infty}^{\bar x_1} (a_1 x_1 + b_1)  \ \tfrac{1}{ \sigma_g} \phi(\tfrac{x_1 -  \mu_1}{\sigma_g}) \mathrm{d}x_1
       +\int_{\bar x_1}^\infty (a_2 x_1 + b_2) \tfrac{1}{ \sigma_g} \phi(\tfrac{x_1 -  \mu_1}{\sigma_g}) \mathrm{d}x_1 \\
      &= (a_2 - a_1) \sigma_g^2 \tfrac{1}{ \sigma_g} \phi(\tfrac{\bar x_1 -  \mu_1}{\sigma_g})
       + (a_1\mu_1 + b_1 )  \Phi(\tfrac{\bar x_1 -  \mu_1}{\sigma_g})
       +   (a_2 \mu_1 + b_2 )  (1-\Phi(\tfrac{\bar x_1 -  \mu_1}{\sigma_g})),
   \end{align}
\end{subequations}
where the last line results from applying Lemma~\ref{lem:int_aff_norm} to each row of the integrated function.
Resubstitution of $a_i$, $b_i$ yields
\begin{equation}
   a_i \mu_1 + b_i = A_i \mu + \bar f_i = f_i(\mu), \quad a_i \sigma_g^2 = A_i\Sigma g,  \quad i=1,2,
\end{equation}
such that
\begin{equation}
      \hat f_\mu(\mu, \Sigma) = (A_2 - A_1) \Sigma g \tilde\phi_g(\bar x; \mu, \Sigma ) 
      + f_1(\mu )  \tilde\Phi_g(\bar x; \mu, \Sigma ) + f_2(\mu) (1 - \tilde\Phi_g(\bar x; \mu, \Sigma )),
\end{equation}
which is identical to \eqref{eq:f_mu}. The variance dynamics \eqref{eq:Sig_mu} follow from Lemma~\ref{lem:Sig_dot_lazy}.
\end{proof}
As discussed in the previous section, this approximation leads to an error even in the piecewise constant case.
However, it is principled in the following sense:
(a) If $A_1=A_2$ and $\bar f_1 = \bar f_2$, i.e., if effectively there is no discontinuity, the corresponding Lyapunov equation for a linear system is recovered \cite{Soederstroem2002}.
(b) With increasing distance of the mean from the switching surface, as measured in terms of the projected standard deviation $\sigma_g$, both $\tilde\phi_g(\bar x_g; \mu, \Sigma)$ and $\tilde\Phi_g(\bar x_g; \mu, \Sigma)$ decay exponentially to $0$ resp. $1$.
Thus, the Lyapunov dynamics for the corresponding mode are recovered asymptotically.

\begin{figure}
   \centering
   \includegraphics[width=\columnwidth]{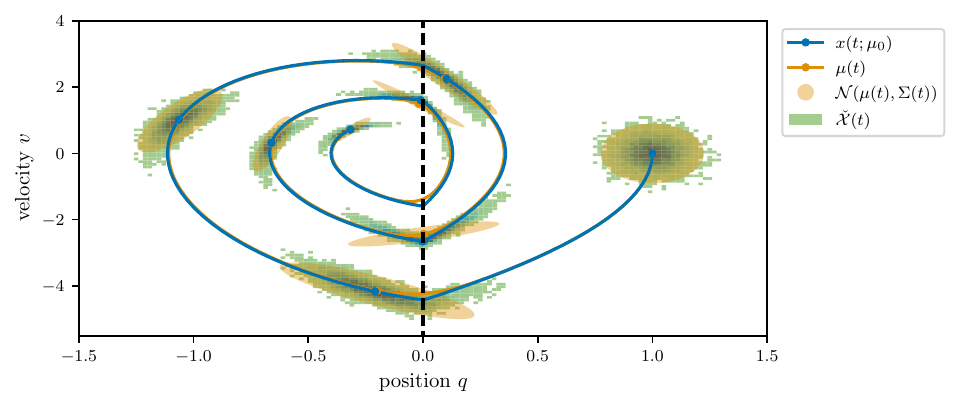}
   \caption{
   Trajectories for the switched system from Example~\ref{ex:bouncing_ball} in state space with histograms of the sampling distribution $\check\calX$ in comparison with the approximating normal distribution shown for selected times.
   }
   \label{fig:lin2D_det}
\end{figure}
\begin{figure}
   \centering
   \includegraphics[width=\columnwidth]{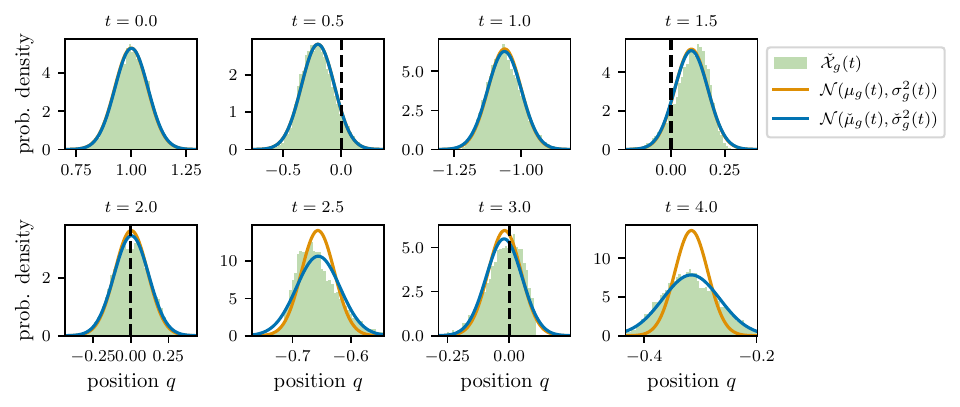}
   \caption{
      The state distributions for the switched system from Example~\ref{ex:bouncing_ball} at selected times projected onto the (negative) switching surface gradient direction. Compared are the sample distribution $\check\calX(t)$, the normal distribution $\calN(\check \mu_g(t), \check\sigma_g^2(t))$ that has the same mean and variance as the sample distribution, and the normal distribution $\calN(\mu_g(t), \sigma_g^2(t))$ resulting from the approximate propagation of mean and variance through \eqref{eq:musig_dot_pw_aff}.}
   \label{fig:lin2D_det_proj}
\end{figure}

\begin{example} \label{ex:bouncing_ball}
   We consider a linear spring/dashpot contact-impact model \cite{Brogliato2016}.
   The system has state $x=(q,v)$ with position $q \in \R$ and velocity $v \in \R$.
   The position $q=0$ corresponds to the system being in contact with a wall but with uncompressed spring.
   For $q<0$ the system is in contact and the spring is compressed such that $\dot v = -k q - c v$, with spring constant $k$ and damping $c$.
   For $q>0$ there is no contact and the spring is uncompressed.
   However, an external force of magnitude $g$ is applied, pushing the mass towards the wall.
   Thus in this case $\dot v = -g$.
   Overall, the dynamics can be written as
   \begin{equation}
      \dot x = \left\{
         \begin{array}{rl} f_1(x), & \psi(x) < 0,  \\f_2(x), & \psi(x) > 0,  \end{array} \right.
          \qquad
         f_1(x) = \begin{bmatrix} x_2 \\ -g \end{bmatrix}, \quad
      f_2(x) = \begin{bmatrix} x_2 \\ -k x_1 - c x_2 \end{bmatrix}, \quad
      \psi(x) = -x_1.
   \end{equation}
   The initial state is distributed as $x_0\sim\calX_0 = \calN(\mu_0, \Sigma_0)$.
   We simulate the system based on the approximate dynamics \eqref{eq:dyn_swi_aff_approx} for mean and covariance yielding a time varying normal distribution $\calN(\mu(t), \Sigma(t))$.
   As proxy for the exact evolution of the distribution we sample $10^4$ points from $\calX_0$ and propagate them based on the nonsmooth dynamics, yielding the sample distribution $\check \calX(t)$ with $\check \calX(0)\approx \calX_0$ and associated mean $\check \mu(t) = \mean_{\check \calX(t)}$ and covariance $\check\Sigma(t) = \cov_{\check \calX(t)}$.
   We also consider the propagation of the original mean $\mu_0$ through the nonsmooth dynamics, denoted by $x(t; \mu_0)$, since for this example it will always correspond to the mode of the exactly propagated distribution.
   The results are shown in Fig.~\ref{fig:lin2D_det}.
   For the sample distribution it can be seen that it becomes strongly deformed -- like a kink -- when crossing a switch but mostly recovers when each sample has passed, cf. Fig.~\ref{fig:lin2D_det} and \ref{fig:lin2D_det_proj}. However some small deformation with respect to a normal distribution remains and accumulates with each switch.
   When looking at the distribution projected onto the switching surface gradient $g$, Fig.~\ref{fig:lin2D_det_proj}, the non-normalcy seems a lot less severe.
   This is relevant since only this projected distribution is used in the dynamics \eqref{eq:musig_dot_pw_aff}.
\end{example}

\section{Uncertainty propagation for general piecewise smooth systems}
\label{sec:pw_smooth}

We return to the general case of \eqref{eq:disc_ode} where the right-hand side of the ODE takes the form 
\begin{equation} \label{eq:f_disc_nl}
   f(x) = \left \{ \begin{array}{rl} f_1(x), & \psi(x)< 0, \\ f_2(x), & \psi(x) > 0,  \end{array} \right.
\end{equation}
with smooth nonlinear components $f_1$, $f_2$, $\psi$.
Given the results for the piecewise affine case from the previous section, we show how they can be applied in this more general setting, based on linearization.
In contrast to the standard linearization based approaches, cf. \eqref{eq:musigdot_standard} and \eqref{eq:musigplus_standard}, which would linearize the full integration map resp. the smooth approximation \eqref{eq:f_smoothed} of the right-hand side, we linearize only the components of \eqref{eq:f_disc_nl} such that the specific structure of the discontinuity is preserved.
We call
\begin{equation}
   f_\mathrm{lin,pw}(x; \mu) = \left \{ \begin{array}{rl} f_1(\mu) + \nabla f_1(\mu)^\top (x - \mu) , & \psi(\mu) + \nabla \psi(\mu)^\top(x - \mu) < 0, \\ f_2(\mu) + \nabla f_2(\mu)^\top (x - \mu), & \psi(\mu) + \nabla \psi(\mu)^\top(x - \mu)  > 0,  \end{array} \right.
\end{equation}
the piecewise linearization of $f$ at $\mu$, which is identical to $f$ if $f$ is piecewise affine.
After this piecewise linearization, the results of Proposition~\ref{prop:musigdot_swi_aff_approx} can in principle be applied, but with some additional dependencies on $\mu$.
The resulting dynamics are
\begin{equation}
   \dot \mu =  \doublehat f_\mu(\mu, \Sigma),
   \quad
   \dot \Sigma = \doublehat f_\Sigma(\mu, \Sigma),
\end{equation}
with 
\begin{subequations} \label{eq:musigdot_pwlinearized}
\begin{align} 
    \doublehat f_\mu(\mu, \Sigma) &\defeq (\nabla f_2(\mu) - \nabla f_1(\mu))^\top \Sigma \nabla \psi(\mu) \frac{1}{\sigma_\psi(\mu, \Sigma)} \phi\left( \frac{-\psi(\mu)}{\sigma_\psi(\mu, \Sigma)} \right) \\
   &\qquad+ f_1(\mu)  \Phi\left( \frac{-\psi(\mu)}{\sigma_\psi(\mu, \Sigma)} \right)+ f_2(\mu)  \left(1 - \Phi\left( \frac{-\psi(\mu)}{\sigma_\psi(\mu, \Sigma)} \right) \right),\\
   \doublehat f_\Sigma(\mu, \Sigma) &\defeq \dpartial{ \doublehat f_\mu(\mu, \Sigma)}{\mu} \Sigma + \Sigma\dpartial{ \doublehat f_\mu(\mu, \Sigma)}{\mu} ^\top,
\end{align}
\end{subequations}
where $\sigma_\psi(\mu, \Sigma) \defeq \sqrt{\nabla \psi(\mu)^\top \Sigma \nabla \psi(\mu)}$ is the standard deviation orthogonal to the switching surface with respect to the linearization of $\psi$ at $\mu$.
Due to the additional dependencies of $\doublehat f_\mu(\mu, \Sigma)$ on $\mu$ the explicit expression for $\doublehat f_\Sigma(\mu, \Sigma)$ is more involved than \eqref{eq:Sig_mu}.
However, if these equations are implemented with a symbolic framework such as CasADi \cite{Andersson2019} there is no need to derive the explicit expressions by hand.
In the case that $f_1$, $f_2$, and $\psi$ are affine, the expressions in 
\eqref{eq:musigdot_pwlinearized} simplify to those of \eqref{eq:musig_dot_pw_aff}.

\section{Stochastic optimal control problem formulation}
\label{sec:stochOCP}
We now demonstrate how the derived dynamics can be used within a stochastic OCP formulation.
After augmenting the dynamics by an argument for the control vector $u(t)\in\R^{n_u}$ -- which for the purpose of integration can be seen as a time dependent parameter -- a rather general OCP can be stated as
\begin{mini!}
   {  \textstyle  u(\cdot),  \mu(\cdot),  \Sigma(\cdot) }
   { \int_{0}^{T} l(\mu(t), u(t)) \mathrm{d}t + L(\mu(T))}
   {\label{eq:stochOCPct}}{}
   \addConstraint{\mu(0)}{=\bar \mu_0, \; \Sigma(0) = \bar\Sigma_0 \label{eq:stochOCPct_iv}}
   \addConstraint{\dot\mu(t)}{= \doublehat f_\mu(\mu(t), \Sigma(t), u(t)),}{t\in[0,T]}
   \addConstraint{\dot\Sigma(t)}{= \doublehat f_\Sigma(\mu(t), \Sigma(t), u(t)),}{t\in[0,T]}
   \addConstraint{0}{\geq g(u(t)),}{t\in[0,T]}
   \addConstraint{0}{\geq h^i(\mu(t)) + \gamma\sqrt{ \nabla h^i(\mu(t))^\top \Sigma(t) \nabla h^i(\mu(t)) },\quad}{ t\in[0,T], \; i=1,\dots,n_h. }
\end{mini!}
Here we consider stage cost $l$ and terminal cost $L$ only for state mean $\mu$ and controls $u$, but it could be straightforwardly extended by a direct cost on the variance $\Sigma$.
The initial mean and covariance are fixed to $\bar\mu_0$ resp. $\bar\Sigma_0$.
For simplicity of notation we consider separate constraints on controls, $g(u) \leq 0$, $g\colon \R^{n_x}\to \R^{n_g}$, and states, $h(x) \leq 0$, $h\colon \R^{n_u}\to \R^{n_h}$, although the formulation can be straightforwardly generalized to combined constraints.
Since $x \in\calN(\mu,\Sigma)$ has unbounded support, the state constraints cannot be strictly enforced for all possible values of $x$.
Instead we use individual chance constraints in each component $h^i(x)$, $i=1,\dots, n_h$, requiring that at each time $t$ (individually) the probability of constraint satisfaction should be at least $p$,
\begin{equation}
   \prob_{x\sim{\calN(\mu, \Sigma)}}\{ h^i(x) \leq 0 \} \geq p.
\end{equation}
After linearization of $h(x)$ at $\mu$ and for $\tfrac{1}{2} < p < 1$, a tractable approximation can be written as
\begin{equation}
   h^i(\mu) + \gamma(p)\sqrt{ \nabla h^i(\mu)^\top \Sigma \nabla h^i(\mu)} \leq 0,
\end{equation}
which consists of the constraint function evaluated at the mean plus an additional backoff term. The backoff term is given by the standard deviation in direction orthogonal to the linearized constraint scaled by $\gamma(p) = \Phi^{-1}(p)$, with $\Phi^{-1}(p)$ the inverse CDF of the standard normal distribution.

The resulting OCP~\eqref{eq:stochOCPct} can then be treated by standard methods of direct optimal control \cite{Rawlings2017} to obtain a nonlinear program (NLP) which can be solved via numerical optimization \cite{Nocedal2006}.

For the examples considered in the following we discretize the time interval $[0,T]$ into an equidistant grid of $N$ intervals, such that the corresponding step length is $h=T/N$, and enumerate the grid points by $k=0,\dots,N$.
Within each interval a constant control $u_k \in\R^{n_u}$ is applied. 
Integration of the approximate mean and covariance dynamics \eqref{eq:musigdot_pwlinearized} over each time interval results in the discretized dynamics
\begin{equation}
   (\mu_{k+1}, \Sigma_{k+1}) = F(\mu_k, \Sigma_k, u_k), \quad k=0,\dots,N-1,
\end{equation}
which here we obtain by applying one step of the explicit Runge-Kutta method of fourth order (RK4) to \eqref{eq:musigdot_pwlinearized}.
Choosing a simplified multiple shooting formulation \cite{Bock1984} the resulting discretized OCP is an NLP of the form
\begin{mini!}
   {  \textstyle \substack{ u_0, \dots, u_{N-1}, \\ \mu_0,\dots,\mu_{N}, \\ \Sigma_0, \dots, \Sigma_N
   }}
   { \sum_{k=0}^{N-1} h l(\mu_k, u_k) + L(\mu_N)}
   {\label{eq:stochOCPdisc}}{}
   \addConstraint{\mu_0}{=\bar \mu_0, \; \Sigma_0 = \bar\Sigma_0}
   \addConstraint{0}{= F(\mu_k, \Sigma_k, u_k) - (\mu_{k+1}, \Sigma_{k+1}),}{k=0,\dots,N-1}
   \addConstraint{0}{\geq g(u_k),}{k=0,\dots,N-1}
   \addConstraint{0}{\geq h^i(\mu_k) + \gamma\sqrt{ \nabla h^i(\mu_k) \Sigma_k \nabla h^i(\mu_k)^\top},\quad \label{eq:stochOCPdisc_constr_x}}{ i=1,\dots,n_h, \; k=1,\dots,N.}
\end{mini!}
While the covariances $\Sigma_k$ are guaranteed to be positive definite at a solution if $\bar \Sigma_0 \succ 0$, they may take arbitrarily indefinite values throughout the solver iterations which can cause problems due to the square root in \eqref{eq:stochOCPdisc_constr_x}.
To avoid this one may use decoupling slack variables, cf. \cite{Messerer2021, Messerer2023}.
In the following examples we formulate the respective OCP of form \eqref{eq:stochOCPdisc} via the Python interface of the symbolic framework CasADi \cite{Andersson2019} and solve them with the interior point method IPOPT \cite{Waechter2006}.

\begin{example}[Quadcopter with wind shadow] \label{ex:quadcopter}
We consider a quadcopter modeled as a mass point and described by state $x=(p_\mathrm{x}, p_\mathrm{y}, v_\mathrm{x}, v_\mathrm{y}) \in\R^4$, i.e., position and velocity in both $\mathrm{x}$ and $\mathrm{y}$ direction. 
The control vector $u=(u_\mathrm{x}, u_\mathrm{y})$ consists of acceleration in the two directions.

We consider two air layers with different wind speeds: 
If $p_\mathrm{y} < 0$ there is no wind due to obstacles blocking its path, and for $p_\mathrm{y} > 0$ there is wind in $p_\mathrm{x}$-direction with $v_\mathrm{w}=-1$.
In both layers we consider air friction in $p_\mathrm{x}$-direction with drag coefficient $d = 0.01$. 
The dynamics are thus 
\begin{equation}
   \dot p_\mathrm{x} = v_\mathrm{x}, \; \dot p_\mathrm{y} = v_\mathrm{y}, \; \dot v_\mathrm{y} = u_\mathrm{y}, \;
\dot v_\mathrm{x}
=
\begin{cases}
u_\mathrm{x} - d \lvert v_\mathrm{x} \rvert v_\mathrm{x} , & p_\mathrm{y} < 0, \\
u_\mathrm{x} - d \lvert v_\mathrm{x} - v_\mathrm{w} \rvert (v_\mathrm{x} - v_\mathrm{w} ), & p_\mathrm{y} > 0.
\end{cases}
\end{equation}
For the considered trajectories we will always have $v_\mathrm{x} > 0$ such that the non-differentiability of the absolute value $\lvert v_\mathrm{x} \rvert$ will not become relevant.
Starting from a fixed initial position the control aim is to maximize the $p_\mathrm{x}$ position in the given time window, $l(x,u) = -p_\mathrm{x} + \epsilon \lVert u \rVert_2^2$, $L(x)=-p_x$, where the controls are slightly regularized with $\epsilon=10^{-5}$.
The quadcopter should stay above the lower bound of $p_\mathrm{y} \geq p_\mathrm{y,min}$ with $p_\mathrm{y,min} = -6$ and in one region the path is blocked by a parabolic obstacle modelled as $p_\mathrm{y} \geq 0.05 (p_\mathrm{x} - 40)^2$.
The controls are constrained as $\lVert u \rVert_\infty \leq 5$.
The initial state is uncertain with $x_0\sim\calN(\bar\mu_0,\bar\Sigma_0)$  where $\bar \mu_0=(0,1,5,0)$ and $\bar \Sigma_0 = \diag(10^{-1},10^{-1},10^{-5},10^{-5})$.
We discretize the time interval into $N=30$ intervals of length $h=0.2$ and solve the corresponding OCP of form \eqref{eq:stochOCPdisc}.
The solution is visualized in Fig.~\ref{fig:quadcopter_traj}.
The quadcopter starts by moving downwards to get into the wind shadow of the obstacle.
Due to the obstacle it has to leave the wind shadow after some time, but does so only as long as necessary.
Since the quadcopter is slowed down more in the upper region this has a rotating effect on the covariance whenever the distribution has nonnegligible support in both layers.

\begin{figure}
\centering
\includegraphics[width=\linewidth]{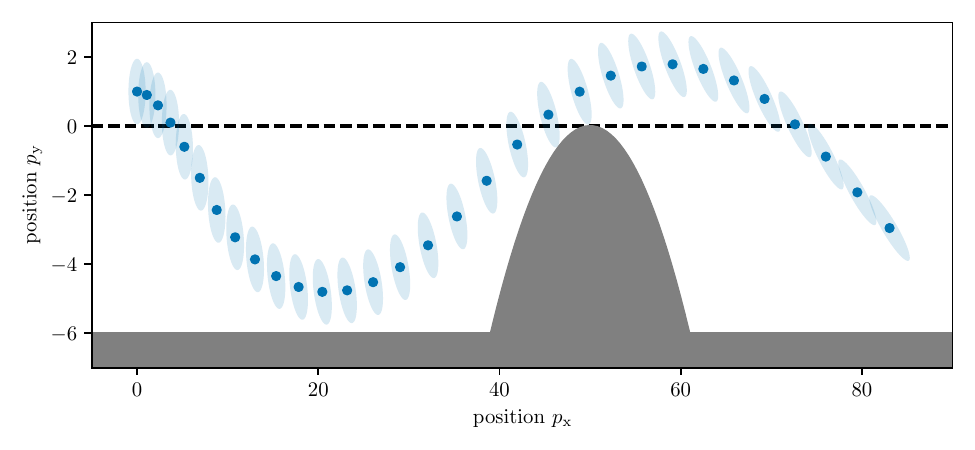}
\caption{Optimal trajectory for the quadcopter from example \ref{ex:quadcopter}.
Shown are the mean and the 99\% confidence region of the normal distribution.
The dashed line indicates the switching surface, the infeasible region is darkly shaded.}
\label{fig:quadcopter_traj}
\end{figure}
\end{example}

\begin{example}[Implicit constraint] \label{ex:black_hole}
We consider a two-dimensional system with state $x=(p_\mathrm{x}, p_\mathrm{y})\in\R^2$.
The controls are $u=(u_\mathrm{x}, u_\mathrm{y})$, constrained by $\lVert u\rVert_\infty \leq u_\mathrm{max}$ with $u_\mathrm{max} = 2$.
In one region of the state space these are directly the velocities of the system, but in the other region there is an additional  vector field, which is in magnitude stronger than the control constraints. 
More specifically we have $f_1(x,u) = u - (4.5, 5) u_\mathrm{max}$, $f_2(x,u)=u$ and $\psi(x) = -p_\mathrm{y} - p_\mathrm{x}^2$.
The control goal is to steer the system to the target position $x_\mathrm{goal}=(6,-2)$ for which we use Huber-like cost terms $l(x, u) = \lVert (x,\varepsilon_{\mathrm{H}})\rVert_2 + \varepsilon_u \lVert u \rVert_2^2$, $L(x) = \lVert (x,\varepsilon_{\mathrm{H}})\rVert_2$, with Huber smoothing $\varepsilon_\mathrm{H} = \sqrt{0.5}$ and slight control regularization $\varepsilon_u = 10^{-5}$.
The initial value is distributed as $x_0\sim\calN(\bar\mu_0,\bar\Sigma_0)$  where $\bar \mu_0=(-6.5, -2)$ and $\bar \Sigma_0 = \tfrac{1}{4} \diag(1,1)$.
We discretize the time interval into $N=15$ intervals of length $h=0.5$ and solve the corresponding OCP of form \eqref{eq:stochOCPdisc}.

The solution is visualized in Fig.~\ref{fig:blackhole-traj} (left), where we additionally sample 50 values from the initial distribution and simulate their trajectories with NOSNOC \cite{Nurkanovic2022b} based on the optimal control trajectory (open loop).
Since in the region with $f_1$ the additional vector field is too strong given the control constraints, $\psi(x)$ acts like an implicit constraint, and the region is mostly avoided.

We compare this to the standard linearization based approach \eqref{eq:musigplus_standard} applied to the smoothed dynamics \eqref{eq:f_smoothed} with smoothing parameter $\sigma=5\cdot10^{-2}$.
Since in these the covariance propagation is based on linearization at the mean, the resulting dynamics only ``see'' the switch if the mean is in its vicinity, independent of the level of uncertainty.
Thus, in the resulting optimal trajectory, cf. Fig~\ref{fig:blackhole-traj} (right), the mean keeps almost no distance from the switching surface, with a significant amount of probability mass overlapping with the second region.
In consequence, when simulating the sample distribution, a large fraction of the samples gets trapped in this region and does not arrive at the target state.

\begin{figure}
   \centering
   \includegraphics[width=\linewidth]{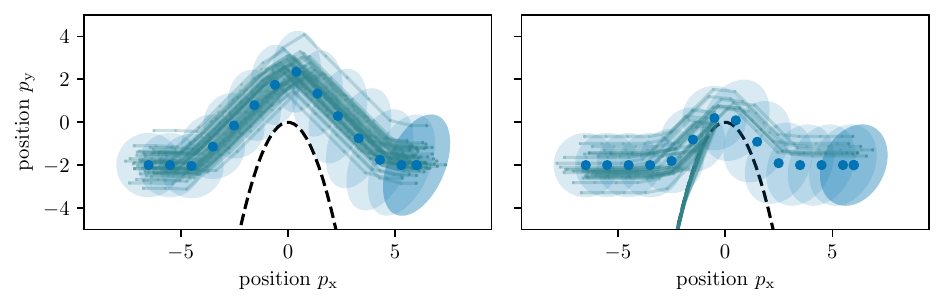}
   \caption{Optimal solutions when using either the approximate dynamics  \eqref{eq:dyn_swi_aff_approx} (left) or the standard linearization based approach  \eqref{eq:musigplus_standard}(right), indicated by blue dots for the mean and a shaded 99\% confidence region.
   The green lines are samples from the initial distribution simulated according to the respective optimal control trajectory (open loop).
   }
   \label{fig:blackhole-traj}
\end{figure}
\end{example}

\section{Conclusions} \label{sec:conclusions}
We derived a method for the approximate propagation of mean and variance through an ODE with discontinuous right-hand side and demonstrated how it can be straightforwardly used in a stochastic OCP formulation.
However, the paper was mostly focused on the derivation of the approximations. 
A formal analysis of the resulting errors would enable a more rigorous theoretical backing of the method.
Further, we only treated the case of two different modes of the right-hand side.
If similar results were derived for the case where the right-hand side has an arbitrary number of modes, the method could be applied in a wider range of situations.
\bibliographystyle{elsarticle-num} 
\bibliography{syscop}

\end{document}

%% file: settings.tex
\usepackage{graphicx}
\usepackage{xcolor}
\usepackage[hidelinks]{hyperref}

\usepackage{amsmath, amssymb, mathtools, amsfonts}
\usepackage[short]{optidef}
\usepackage{booktabs}   % nicer separator lines for tables

\usepackage{amsthm}
\theoremstyle{plain}

\newtheorem{lemma}[]{Lemma}
\newtheorem{corollary}[]{Corollary}
\newtheorem{proposition}[]{Proposition}
\theoremstyle{definition}

\newtheorem{example}[]{Example}

\newcommand{\bmat}{\begin{bmatrix}}
\newcommand{\emat}{\end{bmatrix}}
\newcommand{\dpartial}[2]{\frac{\partial{#1}}{\partial{#2}}}

\newcommand{\dtotal}[2]{\tfrac{\mrm{d} #1}{\mrm{d} #2}}
\newcommand{\dtotaltfrac}[2]{\tfrac{\mrm{d} #1}{\mrm{d} #2}}
\DeclareMathOperator{\diag}{diag}
\DeclareMathOperator{\EXV}{\mathbb{E}}		% expected value
\DeclareMathOperator{\erf}{erf}		
\DeclareMathOperator{\conv}{conv}		
\DeclareMathOperator{\cov}{cov}		
\DeclareMathOperator{\mean}{m}		
\DeclareMathOperator{\prob}{Prob}

\newcommand{\R}{\mathbb{R}}
\newcommand{\SSS}{{\mathbb{S}}}

\newcommand{\calF}{\mathcal{F}}

\newcommand{\calN}{\mathcal{N}}

\newcommand{\calX}{\mathcal{X}}
\newcommand{\calZ}{\mathcal{Z}}
\newcommand{\bigO}{\mathcal{O}}
\newcommand{\smallO}{o}
\newcommand{\defeq}{:=}
\newcommand{\eqdef}{=:}
\newcommand{\eye}{I}  % identity matrix
\newcommand{\expe}{\mathrm{e}}

\newcommand{\mrm}{\mathrm}

\usepackage{accents}
\newlength{\dhatheight}
\newcommand{\doublehat}[1]{%
    \settoheight{\dhatheight}{\ensuremath{\hat{#1}}}%
    \addtolength{\dhatheight}{-0.35ex}%
    \hat{\vphantom{\rule{1pt}{\dhatheight}}%
    \smash{\hat{#1}}}}